\documentclass[11pt]{article}

\usepackage[utf8]{inputenc}
\usepackage{amsmath,mathtools} 
\usepackage{amssymb} 
\usepackage{graphicx} 
\usepackage{enumerate}
\usepackage{array}
\usepackage{anysize}
\usepackage[english]{babel}
\usepackage{textcomp}
\usepackage{makeidx}
\usepackage{longtable}
\usepackage{amsthm}
\usepackage{color}
\usepackage{hyperref}
\usepackage{caption}
\usepackage{subcaption}
\usepackage{tikz}
\usepackage{bbm}
\usepackage{diagbox}
\usepackage{authblk}

\theoremstyle{definition}
\newtheorem{definition}{Definition}
\newtheorem{theorem}{Theorem}
\newtheorem{proposition}{Proposition}
\newtheorem{lemma}{Lemma}

\newtheorem{remark}{Remark}

\newtheorem{assumption}[theorem]{Assumption}
\newtheorem*{mainTheorem}{ Main Theorem}
\newtheorem{examples}[theorem]{ Examples}

\newcommand{\besov}{B^{-d/2-\beta}_{\infty,\infty}}

\newcommand{\Rd}{\mathbb{R}^d}
\newcommand{\Zd}{\mathbb{Z}^d}
\newcommand{\M}{\mathbb{M}}

\numberwithin{equation}{section}

\usepackage{natbib}

\title{Total variation multiscale estimators for linear inverse problems}
\author[1]{Miguel~del~\'Alamo}
\author[1,2]{Axel~Munk}

\affil[1]{Institute for Mathematical Stochastics, University of G\"ottingen \\ Goldschmidtstrasse 7, 37077 G\"ottingen, Germany}
\affil[2]{Max Planck Institute for Biophysical Chemistry, Am Fassberg 11, 37077 G\"ottingen, Germany}
\date{\today}

\begin{document}

\maketitle

\begin{abstract}
{Even though the statistical theory of linear inverse problems is a well-studied topic, certain relevant cases remain open. Among these is the estimation of functions of bounded variation ($BV$), meaning $L^1$ functions on a $d$-dimensional domain whose weak first derivatives are finite Radon measures. The estimation of $BV$ functions is relevant in many applications, since it involves minimal smoothness assumptions and gives simplified, interpretable cartoonized reconstructions. In this paper we propose a novel technique for estimating $BV$ functions in an inverse problem setting, and provide theoretical guaranties by showing that the proposed estimator is minimax optimal up to logarithms with respect to the $L^q$-risk, for any $q\in[1,\infty)$. This is to the best of our knowledge the first convergence result for $BV$ functions in inverse problems in dimension $d\geq 2$, and it extends the results of~\cite{donoho1995nonlinear} in $d=1$. Furthermore, our analysis unravels a novel regime for large $q$ in which the minimax rate is slower than $n^{-1/(d+2\beta+2)}$, where $\beta$ is the degree of ill-posedness: our analysis shows that this slower rate arises from the low smoothness of $BV$ functions. The proposed estimator combines variational regularization techniques with the wavelet-vaguelette decomposition of operators. }
\end{abstract}

\begin{center}
\textbf{Keywords } Inverse problems \ \  Minimax estimation \ \ Total variation \ \ Interpolation inequalities \ \ Wavelet-vaguelette
\end{center}

\begin{center}
\end{center}

\tableofcontents

\section{Introduction}

We consider the problem of estimating a real-valued function $f$ from observations of $Tf$ in a white noise regression model (see e.g.~\cite{Tsybakov})
\begin{equation}
dY(x)=Tf(x)\, dx+\frac{\sigma}{\sqrt{n}}\, dW(x), \ \ \ x\in\mathbb{M}.\label{results:MODEL}
\end{equation}
Here $\mathbb{M}$ denotes an open subset of $\Rd$, $T:\, L^2(\Rd)\rightarrow L^2(\M)$ is a linear, bounded operator, and $dW$ denotes a Gaussian white noise process on $L^2(\M)$ (see Section 2.1.2 in~\cite{gine2015mathematical}). The domain $\mathbb{M}$ on which the data $Y$ is defined is given by the inverse problem under consideration. In case of regression or deconvolution, we may have $\mathbb{M}=\Rd$, while for certain types of tomography we have $\mathbb{M}= \mathbb{R}\times S^{d-1}$~\citep{natterer1986mathematics}, where $S^{d-1}$ denotes the $d$-dimensional unit sphere.
The parameter $\sigma\, n^{-1/2}>0$ serves as a noise level, and we may assume it to be known, since otherwise it can be estimated efficiently (see e.g.~\cite{spokoiny2002variance} or~\cite{munk2005difference}). 
The parametrization $\sigma\, n^{-1/2}$ is motivated by the fact that the white noise model~\eqref{results:MODEL} is an idealization of a nonparametric regression model with $n$ design points and independent normal noise with variance $\sigma^2$ (see e.g.~\cite{brown1996},~\cite{reiss2008} or Section 1.10 in~\cite{Tsybakov}). 
Specifically, the white noise model does not take into account discretization effects, thus simplifying the theoretical analysis (see however Section~\ref{Rem_Discr} for a discussion of this). In the following we will often refer to $n$ as the \textit{sample size}, keeping in mind that this is only an analogy. 

In this setting, our goal is to reconstruct the function $f$ from observations $dY$ and quantify the error made as $n$ grows.
In order to do so, we assume that $f$ is supported inside the unit hypercube $[0,1]^d$. This restriction is somewhat arbitrary: we merely need the support of $f$ to be contained in a compact set. 
Additionally, we make the structural assumption that $f$ is a function of bounded variation, written $f\in BV$. 
\begin{definition}[Functions of bounded variation]\label{def:BVfunc}
The space of functions of bounded variation $BV$ consists of functions $g\in L^1$ whose weak distributional gradient $\nabla g=(\partial_{x_1}g,\cdots,\partial_{x_d}g)$ is a $\mathbb{R}^d$-valued finite Radon measure on $\Rd$. 
The finiteness implies that the bounded variation seminorm of $g$, defined by
\begin{equation*}
 |g|_{BV}:=\sup\bigg\{\int_{\mathbb{R}^d}g(x)\, div(h)(x)\, dx\, \bigg|\, h\in C^1(\mathbb{R}^d;\mathbb{R}^d),\ \|h\|_{L^{\infty}}\leq 1\bigg\},
\end{equation*}
is finite, where $div(h)=\sum_{i=1}^d\partial_{x_i}h_i(x)$ denotes the divergence of the vector field $h=(h_1,\ldots,h_d)$. $BV$ is a Banach space with the norm 
$\|g\|_{BV}:=\|g\|_{L^1}+|g|_{BV}$ (see~\cite{Evans}). Here $C^1(\mathbb{R}^d;\mathbb{R}^d)$ denotes the set of continuously differentiable functions on $\Rd$ taking values on $\Rd$.
\end{definition}

Functions of bounded variation have been used manifold in imaging applications since their introduction in the seminal work by~\cite{ROF}. The reason for their success is that they produce cartoonized reconstructions with sharp edges, which eases interpretability and makes them suitable for applications as diverse as medical imaging, microscopy, astronomy and geology, to mention just a few (see~\cite{scherzer2009variational} and references therein). 
However, in spite of their widespread use, a statistical theory for the estimation of $BV$ functions in inverse problems is still lacking.
 To the best of our knowledge, the only available result for minimax optimal reconstructions of $BV$ functions in inverse problems is~\cite{donoho1995nonlinear}. He introduced the wavelet-vaguelette decomposition (WVD) associated with an operator, and showed that thresholding the WVD yields minimax optimal reconstruction over a range of Besov spaces. His results cover the case of $BV$ functions for $d=1$ and $\beta$-smoothing operators with $\beta\in[0,1/2)$, meaning operators whose singular values behave like $\kappa_j=O(2^{-j\beta})$ as $j\rightarrow\infty$. This includes convolution operators with smooth enough convolution kernels, among others. In contrast, there is no statistical guaranty for estimating $BV$ functions in dimension $d\geq 2$, which covers the very relevant imaging applications.

In this paper we propose an estimator that combines variational regularization with the WVD and multiscale dictionaries. We show that the proposed estimators are minimax optimal up to logarithmic factors for estimating $BV$ functions in any dimension for a variety of inverse problems, including Radon inversion and deconvolution.

\subsection{Multiscale total variation estimation}

We consider the variational estimator
\begin{equation}
\hat{f}_n\in\underset{g\in \mathcal{F}_n}{\textup{ argmin }}|g|_{BV}\ \textup{ subject to } \ \max_{\omega\in\Omega_n}\big|\langle u_{\omega},Tg\rangle-\langle u_{\omega},dY\rangle\big|\leq\gamma_n,\label{result:EST}
\end{equation}
where $\gamma_n$ is a threshold to be chosen, and we minimize over a set of functions $\mathcal{F}_n$ to be specified later. $\Omega_n$ is a finite set of indices, and $\{u_{\omega}\}$ is a vaguelette system associated to the operator $T$,  meaning that
\begin{equation*}
T^*u_{\omega}=\kappa_{\omega}\, \psi_{\omega}\ \ \forall \omega\in\Omega
\end{equation*}
for a wavelet basis $\{\psi_{\omega}\, |\, \omega\in\Omega\}$ and generalized singular values $\kappa_{\omega}$ (see Assumption~\ref{ass:OP} for the details). 
The set $\Omega_n$ depends monotonically on the parameter $n$ in~\eqref{results:MODEL}, which plays the role of the sample size: the larger $n$, the larger the set $\Omega_n$. The reason is that, if the observations $dY$ are very noisy ($n$ small), we do not want to include too many terms in~\eqref{result:EST}, since $\hat{f}_n$ would then be dominated by the noise. Conversely, the smaller the noise level, the more observations we want to include in the data-fidelity in~\eqref{result:EST}, which is then able to extract more information about $f$.

Notice that, by the definition of the vaguelettes, the data-fidelity in~\eqref{result:EST} is actually a constraint on the wavelet coefficients of $g$: they are forced to be close to the wavelet coefficients of the unknown function $f$, up to noise terms. Hence, the data-fidelity in~\eqref{result:EST} amounts to denoising of wavelet coefficients, while the regularization term $|g|_{BV}$ ensures that $\hat{f}_n$ is well-behaved in the $BV$ seminorm.

We deliberately pose the optimization problem~\eqref{result:EST} in constrained form, but emphasize its equivalence to the penalized form
\begin{equation}
\hat{f}\in\underset{g\in \mathcal{F}_n}{\textup{argmin}}\ |g|_{BV}+\lambda\, \max_{\omega\in\Omega_n}\big|\langle u_{\omega},Tg\rangle-\langle u_{\omega},dY\rangle\big|\label{intro:variational}.
\end{equation}
Indeed, both forms are equivalent for suitable parameters $\gamma_n$ and $\lambda$, but these will depend on the data and cannot be transformed easily from one problem to the other. For the penalized formulation~\eqref{intro:variational}, the optimal $\lambda$ could then be chosen in a data-driven way (e.g.~by cross validation~\citep{wahba1977practical} or by a version of Lepskii's balancing principle~\citep{lepskii1991problem}, see e.g.~\cite{mathe2003geometry} in the context of inverse problems). In the constrained formulation~\eqref{result:EST}, the optimal $\gamma_n$ in~\eqref{result:EST} can be chosen in a universal, non data-dependent manner, see equation~\eqref{IP:UnivGamma}. 

To see that, notice that the role of $\gamma_n$ is to decide which functions are allowed for the minimization problem~\eqref{result:EST}: a smaller $\gamma_n$ would yield very few admissible functions, and conversely for larger $\gamma_n$. Since the best reconstruction we can hope for is the true regression function $f$, the optimal $\gamma_n$ would be the one that is large enough to let $f$ be a feasible function, but not larger. In this sense, note that $f$ satisfies the constraint in~\eqref{result:EST} precisely when
\begin{equation}
\max_{\omega\in\Omega_n}\big|\langle u_{\omega},Tf\rangle-\langle u_{\omega},dY\rangle\big|=\max_{\omega\in\Omega_n}\frac{\sigma}{\sqrt{n}}\big|\langle u_{\omega},dW\rangle\big|\leq\gamma_n.\label{intro:QuantileArg}
\end{equation}
Assume for a moment that $u_{\omega}\in L^2$ with $\|u_{\omega}\|_{L^2}=1$ for all $\omega$. Then the left-hand side behaves like the maximum of the absolute value of $\#\Omega_n$ standard normal random variables times $\sigma\, n^{-1/2}$. Consequently, we see that~\eqref{intro:QuantileArg} holds asymptotically with probability one if we choose $\gamma_n\sim\sigma\, n^{-1/2}\, \sqrt{2\log \#\Omega_n}$. This argument can be adapted to the case that the $u_{\omega}$ do not have norm one, as long as their norms remain bounded above and below by positive constants. We remark that this canonical choice of $\gamma_n$ makes the estimator in constrained form~\eqref{result:EST} more convenient from a practical point of view than the one in penalized form~\eqref{intro:variational}.

At this point we can argue why the choice of the data-fidelity term in~\eqref{result:EST} is in a sense optimal: if we had chosen it to be the maximum of \textit{weighted} coefficients, these weights would appear in~\eqref{intro:QuantileArg}, which would then be the maximum of normals with \textit{different} variances. The maximum would hence be dominated by the terms with larger variances, which would lead to overfitting (if small scales dominate) or oversmoothing (if the large scales dominate).

Finally, we argue that the multiscale data-fidelity in~\eqref{result:EST} is in a sense preferable over the $L^2$ data-fidelity, which acts globally on the residuals. Indeed, consider an estimator like~\eqref{result:EST} with an $L^2$ constraint, which would take the form
\begin{equation*}
\sum_{\omega\in\Omega_n}|\langle u_{\omega},Tg\rangle-\langle u_{\omega},dY\rangle|^2\leq \widetilde{\gamma_n}
\end{equation*}
for some $\widetilde{\gamma_n}$, where we used the fact that $\{u_{\omega}\}$ is a frame for $L^2$~\citep{donoho1995nonlinear} to express the $L^2$ norm in terms of the vaguelette coefficients. 
Arguing as above, the optimal $\widetilde{\gamma_n}$ is the one for which the true function $f$ satisfies the constraint. 
Plugging in $g=f$, the left-hand side is a $\chi^2$-distributed random variable, so $\widetilde{\gamma_n}$ should be chosen as $\widetilde{\gamma_n}\sim \sigma^2\, \#\Omega_n/n$. 
The difference between the multiscale and $L^2$ constraints is now apparent: 
\begin{align*}
\textup{ multiscale constraint: } &\ell^{\infty} \textup{ ball of radius } \ \sigma\, n^{-1/2}\sqrt{2\log \#\Omega_n},
\\
L^2 \textup{ constraint: } &\ell^{2}\  \textup{ ball of radius } \ \sigma\, n^{-1/2}\sqrt{\#\Omega_n},
\end{align*}
where both constraints are on the vaguelette domain. 
If we assume that the number of constraints $\#\Omega_n$ grows polynomially in $n$ (see Assumption~\ref{ass:OP}), then the radius in the multiscale constraint tends to zero as $n\rightarrow \infty$, while the radius in the $L^2$ constraint tends to a constant or diverges if $n=O(\#\Omega_n)$. 
Hence, the multiscale constraint set is much smaller for $n$ large, and we expect the multiscale data-fidelity to produce more faithful reconstructions.

 Before we turn to the discussion of the convergence properties of $\hat{f}_n$, let us discuss two potential limitations of our approach. First, not every operator $T$ has an associated vaguelette system $\{u_{\omega}\}$, as we use in~\eqref{result:EST}. In fact, only reasonably homogeneous operators admit such a system (see~\cite{donoho1995nonlinear}). However, for our theory we do not need the whole generality of the WVD (see Assumption~\ref{ass:OP}), and many relevant operators such as the Radon transform, convolution or integration satisfy our assumptions (see Examples~\ref{result:Examples} below). 

The second limitation concerns the numerical solution of the optimization problem in~\eqref{result:EST}, which in general is a non-smooth, high-dimensional optimization problem (since $n$ and $\#\Omega_n$ might be large). While classical techniques such as interior point methods~\citep{nesterov1994interior} find their limitations here, the computation of~\eqref{results:MODEL} is meanwhile feasible due to recent progress in convex optimization, e.g.~in primal-dual methods~\citep{chambolle} and accelerations thereof~\citep{malitsky2018first}, or semismooth Newton methods with the path-following technique~\citep{clason}. We will not elaborate on this issue further and postpone this to future work.

\subsection{Main result}

The main result of this paper states that the estimator~\eqref{result:EST} is minimax optimal (up to logarithmic factors) for estimating $BV$ functions in any dimension for certain inverse problems. In order to formulate our result we need to introduce some notation. 
For $L>0$ define the intersection of a $BV$-ball of radius $L$ with an $L^{\infty}$-ball as
\begin{equation}
BV_L:=\big\{g\in BV\cap \mathcal{D}(T)\, \big|\, |g|_{BV}\leq L, \ \ \|g\|_{L^{\infty}}\leq L,\ \ \textup{supp }g\subseteq [0,1]^d\big\},\label{setBVL}
\end{equation}
where $\mathcal{D}(T)\subset L^2$ denotes the domain of the operator $T$. The reason for the support condition in~\eqref{setBVL} is the following: since we only have a finite amount of information, we cannot hope to recover a function with infinite support. The restriction
to the unit cube is in a sense arbitrary: any regular enough compact set would do.

For given $d$, $\beta\geq 0$ and $q\in[1,\infty]$, define the number
\begin{equation}
\vartheta_{q,\beta}:=\begin{cases}
\frac{1}{d+2\beta+2} & \textup{ for } q\leq 1+2/(d+2\beta)
\\
\frac{1}{q \, (d+2\beta)} & \textup{ for } q> 1+2/(d+2\beta).
\end{cases}\label{results:vartheta2}
\end{equation}

Our main result (Theorems~\ref{results:Main_thm} and~\ref{results:Thm2}) can be stated informally as follows.

\begin{mainTheorem}[Informal]
For $d\in\mathbb{N}$ and $\beta\geq 0$, let $T$ have a WVD with singular values behaving as $\kappa_j=2^{-j\beta}$ (see Assumption~\ref{ass:OP} in Section~\ref{sect:results}). Let the threshold $\gamma_n$ be as in~\eqref{IP:UnivGamma} for $\kappa>\kappa^*$ depending on $T$ and $d$ only. Then the estimator $\hat{f}_n$ attains the \textit{minimax optimal} rate of convergence over $BV_L$ up to a logarithmic factor,
\begin{equation}
     \sup_{f\in BV_L}\mathbb{E}\big[\|\hat{f}_n-f\|_{L^q}\big]\leq C_L\, n^{-\vartheta_{q,\beta}}\, (\log n)^{3-\min\{2,d\}}\label{Intro:Conver1}
\end{equation}
for $n$ large enough, for any $q\in\big[1,\infty\big)$, any $L>0$ and a constant $C_L>0$ independent of $n$, but dependent on $L$, $\sigma$, $d$ and $T$. 
\end{mainTheorem}

The convergence rate in~\eqref{Intro:Conver1} is indeed minimax optimal over the class $BV_L$ up to the logarithmic factor, as it is the optimal rate over the smaller class of bounded Besov $B^1_{1,1}$ functions, see Theorem~\ref{results:Thm2} and Section~\ref{sect:Notation} for the definition of Besov spaces. The minimax rate $n^{-1/(d+2\beta+2)}$ is well-known for inverse problems when $q\leq 1+2/(d+2\beta)$ (see e.g.~\cite{cavalier2011inverse}). In contrast, the "slow" regime with rate $n^{-\frac{1}{q(d+2\beta)}}$ for $q>1+2/(d+2\beta)$ has been observed for the specific case $\beta=0$ in density estimation~\citep{goldenshluger} and nonparametric regression (\cite{lepskii2015} and~\cite{mTV}) when estimating over anisotropic Nikolskii classes $\mathbb{N}^s_p$ and Besov classes $B^s_{p,t}$ with $s<d/p$. Moreover, the slow regime explains the recently observed phase transition in the $L^2$ minimax risk for estimating discretized $TV$ functions in the particular case $\beta=0$, see~\cite{sadhanala2016total}. Our result extends these findings to linear inverse problems. 

The proof of the minimax optimality of that rate is based on the construction of a set of alternatives in the smaller space $B^1_{1,1}\subset BV$. Interestingly, the set of alternatives that attains the minimax rate is neither sparse nor dense: it presents blocks of signals at different locations. We conjecture that only estimators that incorporate a form of spatial adaptation can be minimax optimal in this regime, as the ones proposed in~\cite{lepskii2015}, in~\cite{mTV} and in the present paper.

The proof of the Main Theorem is based on an upper bound on the $L^q$-risk with an interpolation inequality in terms of the $BV$ norm and a Besov norm of negative smoothness,
\begin{equation}
   \|\hat{f}_n-f\|_{L^q}\leq C \|\hat{f}_n-f\|_{B^{-d/2-\beta}_{\infty,\infty}}^{\frac{2}{d+2\beta+2}}\|\hat{f}_n-f\|_{BV}^{\frac{d+2\beta}{d+2\beta+2}}\hspace{0.5cm} \forall g\in\besov\cap BV\label{Intro:DummyInt}
\end{equation}
for any $q\in\big[1,\frac{d+2\beta+2}{d+2\beta}\big]$, $d\geq 2$. See Section~\ref{sect:Notation} for the definition of Besov spaces. This inequality follows from a result by~\cite{cohen2003harmonic}, proved by an analysis of the wavelet coefficients of $BV$ functions. Since we have $\hat{f}_n\in BV$ by construction, the $BV$ norm in~\eqref{Intro:DummyInt} is easily bounded by a constant. On the other hand, the Besov norm can be related to the constraint in the right-hand side of~\eqref{result:EST}, and some analysis yields the bound
\begin{equation*}
\|\hat{f}_n-f\|_{B^{-d/2-\beta}_{\infty,\infty}}\leq C\,n^{-1/2}\, \log n
\end{equation*}
with high probability. Plugging this expression in~\eqref{Intro:DummyInt}, we get the desired bound for the $L^q$-risk. The bound is extended to $q>1+\frac{2}{d+2\beta}$ using H\"older's inequality. For $d=1$ we proceed analogously with some modifications. See Section~\ref{Sec:Main_Proof} for a complete proof.

\subsection{Related work}

Notwithstanding the success of $BV$ functions in imaging applications (see~\cite{ROF} for the first reference), there are very few works that analyze the estimation of $BV$ functions in a statistical setting. In nonparametric regression ($T=id$), classical results (\cite{mammen1997} and~\cite{donoho1998minimax}) established minimax optimality results for estimation in dimension $d=1$, and recently a class of multiscale variational estimators was shown to perform optimally in any dimension~\citep{mTV}, whose approach we generalize here to $T\neq id$. In statistical inverse problems, the only work proving minimax optimal convergence rates for the estimation of $BV$ is, to the best of our knowledge,~\cite{donoho1995nonlinear}. He shows that thresholding of the WVD is minimax optimal over a range of Besov spaces $B^{s}_{p,t}$ and for a class of $\beta$-smoothing inverse problems. In the case relevant for $BV$ ($s=p=1$), the minimax optimality holds for the range $\beta<1-d/2$, i.e.~for $\beta$ smoothing operators in dimension $d=1$ and $\beta\in[0,1/2)$. The present work is hence an improvement, since we do not impose any limitation on $\beta$ nor on the dimension $d$. On the other hand, we get a suboptimal logarithmic factor in~\eqref{Intro:Conver1}, while~\cite{donoho1995nonlinear} achieves the exact optimal rate.

At a technical level, our work is inspired by several sources. We have already mentioned~\cite{donoho1995nonlinear}, who introduced the WVD as a means for using wavelet methods in inverse problems (see also~\cite{abramovich1998wavelet} for a variant of the WVD, and~\cite{candes2002recovering} for a refined approach to Radon inversion). Besides these works, there have been several approaches that implicitly use the WVD idea. We refer to~\cite{schmidt2013multiscale} and~\cite{proksch2018multiscale} for hypothesis testing in inverse problems, where multiscale dictionaries adapted to the operator $T$ are employed. Another source of inspiration for our work are nonparametric methods that combine variational regularization with multiscale dictionaries. We refer exemplarily to~\cite{CandesGuo},~\cite{dong2011automated},~\cite{frick2012} and~\cite{frick2013statistical} for an empirical analysis of such methods in simulations. Moreover, the proof of our main result is based on the above mentioned interpolation technique: an interpolation inequality of the form~\eqref{Intro:DummyInt} is used to relate the risk functional, the regularization functional and the data-fidelity. This technique was used by~\cite{nemirovskii1985nonparametric} and~\cite{grasmair2015} 
for estimating Sobolev functions, using an extension of the Gagliardo-Nirenberg interpolation inequalities~\citep{nirenberg1959}, and by~\cite{mTV} for the estimation of $BV$ functions, employing a generalization thereof (\cite{Meyer},~\cite{cohen2003harmonic}). In that sense, the present work combines the tools developed in~\cite{mTV} with the WVD from~\cite{donoho1995nonlinear}, and it generalizes both results.

\subsection*{\em Organization of the paper}

The rest of the paper is organized as follows. In Section~\ref{sect:results} we state our assumptions and main theorems, and give their proofs. We also discuss the particular inverse problems of deconvolution and Radon inversion. The proofs of auxiliary results are given in Section~\ref{sect:Proofs}.

\section{Results}\label{sect:results}

\subsection{Notation}\label{sect:Notation}

\textit{Basic notation.} We denote the Euclidean norm of a vector $v=(v_1,\ldots,v_d)\in\Rd$ by $|v|:=\big(v_1^2+\cdots+v_d^2\big)^{1/2}$. For a real number $x$, define $\lfloor x\rfloor:=\textup{max}\big\{m\in\mathbb{Z}\, \big|\, m\leq x\big\}$ and $\lceil x\rceil:=\textup{min}\big\{m\in\mathbb{Z}\, \big|\, m>x\big\}$. The cardinality of a finite set $X$ is denoted by $\# X$. We say that two sequences $a_n$ and $b_n$, $n\in\mathbb{N}$, grow at the same rate, written $a_n\asymp b_n$, if there are (potentially zero) constants $c_1,c_2\geq 0$ such that $c_1a_n\leq b_n\leq c_2 a_n$ for all $n\in\mathbb{N}$. Finally, we denote by $C$ a generic positive constant that may change from line to line. 
\\

\textit{Wavelet bases.} Let $\{\psi_{j,k,e}\, |\, (j,k,e)\in\Lambda\}$ denotes a wavelet basis of $L^2(\Rd)$ formed by tensorization of Daubechies wavelets~\citep{daubechies1992ten} with $D$ continuous partial derivatives and whose mother wavelet has $R$ vanishing moments. Here $j\geq 0$ is a scale index, $k\in\Zd$ is a position index, and $e=(e_1,\ldots,e_d)\in\{0,1\}^d$ denotes whether $\psi_{j,k,e}$ is a mother or a father wavelet along each coordinate. 
We recall that one-dimensional Daubechies wavelets with $R$ vanishing moments have support of size $2R-1$ (with respect to the Lebesgue measure) and are $\lfloor 0.18 \cdot (R-1)\rfloor$ times continuously differentiable (see Theorem 4.2.10 in~\cite{gine2015mathematical}). 
A $D$-smooth wavelet basis formed by tensorization of one-dimensional Daubechies wavelets needs to satisfy $R=1+6D$ in order to have $\lfloor 0.18 \cdot 6\cdot D\rfloor>D$ continuous derivatives. Consequently, the mother and father wavelets have support of size $(12\,D+1)^d$.

In this work we will mainly deal with functions $g$ supported inside the unit cube, $\textup{supp } g\subseteq [0,1]^d$. We will use their wavelet expansion intensively, so let us introduce the set of wavelets with nonzero overlap with the unit cube
\begin{equation}
\Omega=\{(j,k,e)\in\Lambda\, |\, \textup{supp }\psi_{j,k,e}\cap (0,1)^d\neq\emptyset\}.\label{results:omega}
\end{equation}
For each $n\in\mathbb{N}$, $n\geq 2$, let
\begin{equation*}
\Omega_n:=\{(j,k,e)\in\Omega\, |\, j\leq \lceil d^{-1}\log n \rceil\}
\end{equation*}
denote the set of indices of wavelets at scales rougher that $\lceil d^{-1}\log n\rceil$. Since the wavelets at scale $j=0$ have support of size $(12\,D+1)^d$, it follows that there are $O(2^{(j+1)d})$ indices $(j,k,e)\in\Omega$ at level $j$, and hence the cardinality of $\Omega_n$ is of the order $\#\Omega_n\asymp 2^{d\lceil d^{-1}\log n\rceil}\asymp n$.
\\

\textit{Besov spaces.} Let $\{\psi_{j,k,e}\}$ be a wavelet basis with $D$ continuous partial derivatives and whose mother wavelet has $R$ vanishing moments. For $p,q\in[1,\infty]$ and $s\in\mathbb{R}$ with $\min\{R,D\}>|s|$, the Besov space $B^s_{p,q}(\Rd)$ consists of all functions (or distributions) $g$ with finite Besov norm
\begin{equation}
\|g\|_{B^{s}_{p,q}}:=\bigg(\sum_{j\geq 0} 2^{jq\big(s+\frac{d}{2}-\frac{d}{p} \big)}\bigg(\sum_{k\in\Zd}\sum_{e\in\{0,1\}^d}|\langle\psi_{j,k,e},g\rangle|^p\bigg)^{q/p}\bigg)^{1/q}.
\end{equation}
We refer to Section 4.3 in~\cite{gine2015mathematical} for more details.
\\

Finally, we define the Fourier transform of a function $g\in L^1(\Rd)$ by
\begin{equation*}
 \mathcal{F}[g](\xi):=\int_{\Rd}g(x)\, e^{-i \xi \cdot x}\, dx, \ \ \ \xi\in\Rd.
\end{equation*}
The Fourier transform can be extended as an operator to $L^2$ and, by duality, to 
distributions $\mathcal{D}^*(\Rd)$ (see e.g.~Section 4.1.1 in~\cite{gine2015mathematical}).

\subsection{Main results}\label{sect:Results}

We make the following assumptions on the operator $T$.

\begin{assumption}\label{ass:OP}
Let $T: L^2(\Rd)\rightarrow L^2(\M)$ denote a bounded, linear operator. 
For $\beta\geq 0$, assume that the following hold:
\begin{itemize}
\item there is a wavelet basis $\{\psi_{j,k,e}\, \big|\, (j,k,e)\in\Lambda\}$ of $L^2(\Rd)$ (see Section~\ref{sect:Notation}) with $D$ continuous partial derivatives and whose mother wavelet has $R$ vanishing moments, such that $\min\{R,D\}>\max\{1,d/2+\beta\}$;
\item there is a set of functions $\{u_{j,k,e}\, \big|\, (j,k,e)\in\Lambda\}\subset L^2(\M)$, which we call \textit{vaguelette system}, s.t.
\begin{align}
&T^* u_{j,k,e} = \kappa_j\, \psi_{j,k,e} \ \ \ \forall (j,k,e)\in\Lambda, \label{results:Ass1}
\end{align}
with singular values $\kappa_{j}=2^{-j\beta}$. Furthermore, the vaguelettes satisfy
\begin{align*}
&c_1\leq \|u_{\omega}\|_{L^2}\leq c_2 \ \ \ \forall \omega\in\Lambda
\end{align*}
for some real constants $0<c_1<c_2$.
\end{itemize} 
\end{assumption}

We remark that a vaguelette system as constructed in~\cite{donoho1995nonlinear} is a frame. However, we will not need that property in the following.

\begin{remark}\label{results:Remark1}
\quad
\begin{itemize}
\item[a)] Assumption~\ref{ass:OP} is slightly weaker than assuming that the operator $T$ has a wavelet-vaguelette decomposition (WVD)~\citep{donoho1995nonlinear}. In the following we nevertheless call $\{u_{j,k,e}\}$ a vaguelette system for simplicity.
\item[b)]  As remarked in Section~\ref{sect:Notation}, we will only need the wavelets with nonzero overlap with the unit cube, which we index by the set $\Omega$ in~\eqref{results:omega}. In the following we index the vaguelettes accordingly.
\item[c)] The condition $\min\{R,D\}>\max\{1,d/2+\beta\}$ is necessary for ensuring that the norms of the Besov spaces $B^{-d/2-\beta}_{\infty,\infty}$ and $B^1_{p,q}$, $p,q\in[1,\infty]$, can be expressed in terms of wavelet coefficients with respect to the basis $\{\psi_{j,k,e}\}$ (see Section~\ref{sect:Notation}, or Section 4.3 in~\cite{gine2015mathematical}).
\item[d)] Let $\{\psi_{j,k,e}\}$ be a smooth enough wavelet basis. Then condition~\eqref{results:Ass1} implies that the inverse problem~\eqref{results:MODEL} is mildly ill-posed with degree of ill-posedness $\beta$. 
\end{itemize}
\end{remark}

\begin{examples}\label{result:Examples}
We list here some examples of operators satisfying Assumption~\ref{ass:OP}.
\begin{itemize}
\item[a)] The integration operator
\begin{equation*}
Tg(x):=\int_{-\infty}^{x}g(y)\, dy, \ \ \ x\in\mathbb{R}.
\end{equation*}
Its domain consists of functions $g$ such that $|\xi|^{-1}\mathcal{F}[g](\xi)\in L^2(\mathbb{R})$, where $\mathcal{F}$ denotes the Fourier transform. The vaguelettes are given by derivatives and integrals of the wavelet basis, and the critical values are $\kappa_j=2^{-j}$. Fractional integration, iterated integration and higher dimensional integrals also define operators satisfying Assumption~\ref{ass:OP}. We refer to~\cite{donoho1995nonlinear} for more details.
\item[b)] The Radon transform, which maps a function $g$ to 
\begin{equation}
Tg(r,\theta):=\int_{\{x\, \cdot\, \theta = r\}} g(x)\, dx, \ \ r\in\mathbb{R}, \ \ \theta\in S^{d-1}, \label{examples:Radon}
\end{equation}
where the integral is taken over the hyperplane defined by vectors $x$ satisfying $x\cdot\theta=r$. See Section~\ref{sect:ExampRad} for more details on how our estimator~\eqref{result:EST2} works for the Radon transform.
\item[c)] The convolution operator
\begin{equation*}
Tg(x):=\int_{\Rd}K(x-y)g(y)\, dy
\end{equation*}
for a regular enough kernel $K\in L^1(\Rd)$ satisfies Assumption~\ref{ass:OP}. See Section~\ref{sect:ExampConv} for the details.
\item[d)] The identity operator, in which case we are in the white noise regression model. We can take $\{\psi_{j,k,e}\}$ to be a smooth enough wavelet basis, and the estimator~\eqref{result:EST2} reduces (with minor modifications) to the multiscale total variation estimator analyzed in~\cite{mTV}. Besides some differences in the setting (here we estimate compactly supported functions, there periodic ones), the convergence rate that we prove here coincides for $\beta=0$ with the result in~\cite{mTV}.
\end{itemize}
More generally, operators satisfying a certain homogeneity condition with respect to dilations have a WVD (see~\cite{donoho1995nonlinear} for a general result). 
Conversely, Assumption~\ref{ass:OP} is in general not satisfied for operators $T$ with a strong preference for a particular scale. An extreme example is convolution with a kernel whose Fourier transform has compact support. In that case, the equation $T^*u_{j,k,e}=\kappa_j\psi_{j,k,e}$ does not admit solutions $u_{j,k,e}$ for compactly supported wavelets $\psi_{j,k,e}$. 
\end{examples}

In this setting, we define our estimator as follows.

\begin{definition}\label{results:definition}
Let the observations $dY$ follow the model~\eqref{results:MODEL}, and let the operator $T$ satisfy Assumption~\ref{ass:OP} with a vaguelette system $\{u_{j,k,e}\}$. We denote
\begin{equation}
\hat{f}_n\in\underset{g\in \mathcal{F}_n}{\textup{ argmin }}|g|_{BV}\ \textup{ subject to } \ \max_{\omega\in\Omega_n}\big|\langle u_{\omega},Tg\rangle-\langle u_{\omega},dY\rangle\big|\leq\gamma_n,\label{result:EST2}
\end{equation}
as the \textit{multiscale total variation estimator} for the operator $T$. In~\eqref{result:EST2} we minimize over the set
\begin{equation}
\mathcal{F}_n=\{g\in BV\cap L^{\infty}\, \big|\, \|g\|_{L^{\infty}}\leq \log n, \ \textup{supp } g\subseteq [0,1]^d\}.\label{results:Fn}
\end{equation}
We use the convention that, whenever the feasible set of the problem~\eqref{result:EST2} is empty (which happens with vanishing probability as $n$ grows, see Remark~\ref{results:Rem1}), the estimator $\hat{f}_n$ is set to zero. 
\end{definition}

The reason for requiring the support to be inside the closed unit cube in~\eqref{results:Fn} is to make the set $\mathcal{F}_n$ closed. This is important for ensuring existence of a minimizer in~\eqref{result:EST2} as the limit of a minimizing sequence. 

Concerning the choice of the threshold $\gamma_n$, let $\sigma>0$ be as in~\eqref{results:MODEL}, and let $c_2$ be the upper bound in Assumption~\ref{ass:OP}. For $\kappa>0$, we choose
\begin{equation}
\gamma_n=\kappa\, c_2\, \sigma\, \sqrt{\frac{2\log\#\Omega_n}{n}}.\label{IP:UnivGamma}
\end{equation}
Notice that the upper bound $c_2$ can be computed from the dictionary, as we do in the examples in Section~\ref{sect:Examp}.

\begin{remark}\label{results:Rem1}
Let us discuss the feasible set of the problem~\eqref{result:EST2}, which consists of the constraints
\begin{equation}
\max_{\omega\in\Omega_n}\big|\langle u_{\omega},Tg\rangle-\langle u_{\omega},dY\rangle\big|\leq\gamma_n, \ \ \ \|g\|_{L^{\infty}}\leq \log n, \ \ \  \textup{supp }g\subseteq [0,1]^d\label{results:constraints}.
\end{equation}
Here we assume that the observations $dY$ arise from a function $f\in BV_L$, as defined in~\eqref{setBVL}. 
By Proposition~\ref{app:GoodEvent} below and the choice~\eqref{IP:UnivGamma} for $\gamma_n$, the probability that the true regression function $f$ satisfies the first constraint in~\eqref{results:constraints} is not smaller than $1-O(n^{1-\kappa^2})$. As long as $f$ satisfies the first constraint in~\eqref{results:constraints}, it also satisfies the others for $n$ large enough ($n\geq e^L$), since we assume that $f\in BV_L$. 
As a consequence, the feasible set of~\eqref{result:EST2} is nonempty with probability of the order $1-O(n^{1-\kappa^2})$. Hence, we will see that the caveat in Definition~\ref{results:definition} about the feasible set does not play a decisive role for the convergence properties of $\hat{f}_n$.
\end{remark}

\begin{theorem}\label{results:Main_thm}
For $d\in\mathbb{N}$, let $T$ satisfy Assumption~\ref{ass:OP} with $\beta\geq 0$. Assume the model~\eqref{results:MODEL} with $f\in BV_L$ for some $L>0$. For $q\in \big[1,\infty\big)$, let $\vartheta_{q,\beta}$ be as in~\eqref{results:vartheta2}.
\begin{itemize}
 \item[a)] Let $\gamma_n$ be as in~\eqref{IP:UnivGamma} with $\kappa>1$. Then for any $n\in\mathbb{N}$ with 
 $n\geq e^L$, the estimator $\hat{f}_{n}$ in~\eqref{result:EST2} with parameter $\gamma_n$ satisfies
 \begin{equation}
    \sup_{f\in BV_L}\|\hat{f}_{n}-f\|_{L^q}\leq C\, n^{-\vartheta_{q,\beta}}\, (\log n)^{3-\min\{d,2\}}\label{IP:Gen_Conv}
 \end{equation}
for any $q\in[1,\infty)$ with probability at least $1-\big(\#\Omega_n\big)^{1-\kappa^2}$, for a constant $C>0$ independent of $n$, but depending on $L$, $\sigma$ and $d$.
\item[b)] Under the assumptions of part a), if $\kappa^2>1+1/(d+2\beta+2)$, then
 \begin{equation}
    \sup_{f\in BV_L}\mathbb{E}\big[\|\hat{f}_{n}-f\|_{L^q}\big]\leq C\, n^{-\vartheta_{q,\beta}}\, (\log n)^{3-\min\{d,2\}}\label{IP:Gen_Conv_Exp}
 \end{equation}
 holds for any $q\in[1,\infty)$, $n$ large enough and a constant $C>0$ independent of $n$. 
\end{itemize}
\end{theorem}

Theorem~\ref{results:Main_thm} gives an upper bound for the expected error of $\hat{f}_n$. We now prove a matching lower bound. For that, we assume that $T$ satisfies
\begin{equation}
\|T\psi_{j,k,e}\|_{L^2}\leq c\, 2^{-j\beta} \ \ \forall (j,k,e)\in\Omega\label{lowBound:OP}
\end{equation}
for a constant $c>0$, where $\{\psi_{j,k,e}\}$ is a wavelet basis of compactly supported wavelets. We remark that~\eqref{lowBound:OP} is satisfied by any operator with a WVD (see~\cite{donoho1995nonlinear}).

\begin{theorem}\label{results:Thm2}
Consider the setting of Theorem~\ref{results:Main_thm}, and assume that the operator $T$ admits a WVD. Then the minimax $L^q$-risk over $BV_L$ given observations~\eqref{results:MODEL} is lower bounded by $c\, n^{-\vartheta_{q,\beta}}$. In particular, the estimator~\eqref{result:EST2} is asymptotically minimax optimal up to logarithmic factors for estimating functions $f\in BV_L$, $L>0$, with respect to the $L^q$-risk, for any $q\in\big[1,\infty\big)$.
\end{theorem}

\subsection{Proofs of the main theorems}\label{Sec:Main_Proof}

\subsubsection{Proof of Theorem~\ref{results:Main_thm}}

The proof of Theorem~\ref{results:Main_thm} relies on a variant of an interpolation inequality prove by~\cite{cohen2003harmonic}.

\begin{proposition}\label{ParticIntd=1}
For $d\in\mathbb{N}$ and $\beta\geq 0$, let $q^*:=1+2/(d+2\beta)$. 
\begin{itemize}
\item[a)] If $q^*\leq 2$, there is a constant $C>0$ such that
\begin{equation*}
 \|g\|_{L^q}\leq C\, \|g\|_{B^{-d/2-\beta}_{\infty,\infty}}^{\frac{2}{d+2\beta+2}}\|g\|_{BV}^{\frac{d+2\beta}{d+2\beta+2}}
\end{equation*}
holds for any $q\in[1,q^*]$ and any $g\in B^{-d/2-\beta}_{\infty,\infty}\cap BV$ with supp $g\subseteq [0,1]^d$.
\item[b)] If $q^*>2$, then there is a constant $C>0$ such that for any $n\in\mathbb{N}$ we have
\begin{equation*}
  \|g\|_{L^q}\leq C (\log n )\, \|g\|_{B^{-d/2-\beta}_{\infty,\infty}}^{\frac{2}{d+2\beta+2}}\|g\|_{BV}^{\frac{d+2\beta}{d+2\beta+2}}+C\, n^{-1}\, \|g\|_{L^{\infty}}^{\frac{2}{d+2\beta+2}}\, \|g\|_{BV}^{\frac{d+2\beta}{d+2\beta+2}}
 \end{equation*}
for any $q\in[1,q^*]$ and any $g\in L^{\infty}\cap BV$ with supp $g\subseteq[0,1]^d$.
\end{itemize}
\end{proposition} 
The proof of Proposition~\ref{ParticIntd=1} is given in Section~\ref{sect:Proofs} below. Define the event
\begin{equation}
\mathcal{A}_n:=\bigg\{\max_{(j,k,e)\in\Omega_n}\bigg|\int_{\M} u_{j,k,e}(x)\, dW(x)\bigg| \leq \frac{\sqrt{n}}{\sigma}\, \gamma_n \bigg\},\label{sketch:GoodEvent}
\end{equation}
where $\{u_{j,k,e}\}$ is the vaguelette system from Assumption~\ref{ass:OP}. 

\begin{proposition}\label{app:GoodEvent}
 Let $\{u_{j,k,e}\}$ be a vaguelette system as described in Assumption~\ref{ass:OP}. For any $n\in\mathbb{N}$ we have
 \begin{equation*}
  \mathbb{P}\bigg(\max_{(j,k,e)\in\Omega_n}\bigg|\int_{\M}u_{j,k,e}(x)\, dW(x)\bigg|\geq c_2\, t\bigg)\leq \#\Omega_n\,  e^{-t^2/2}
 \end{equation*}
for any $t\geq 0$, where $c_2$ is the upper bound in Assumption~\ref{ass:OP}.
\end{proposition}

\begin{proof}
 The random variables $\epsilon_{j,k,e}:=c_2^{-1}\int_{\M}u_{j,k,e}(x)\, dW(x)$ are normal with variance smaller than one, since 
 $\|u_{j,k,e}\|_{L^2}\leq c_2$ by the inequality in Assumption~\ref{ass:OP}. By the union bound we have
 \begin{align*}
  \mathbb{P}\big(\max_{(j,k,e)\in\Omega_n}|\epsilon_{j,k,e}|\geq t\big)&\leq \sum_{(j,k,e)\in\Omega_n}\mathbb{P}(|\epsilon_{j,k,e}|\geq t)
 \end{align*}
for any $t\geq 0$, and the probability in the right-hand side can be bounded as
\begin{align*}
\mathbb{P}(|\epsilon_{j,k,e}|\geq t)\leq 2\int_t^{\infty}e^{-x^2/2}\, \frac{dx}{\sqrt{2\pi}} \leq 2\, e^{-t^2}\int_t^{\infty} e^{-x^2/2+xt}\, \frac{dx}{\sqrt{2\pi}}=e^{-t^2/2}.
\end{align*}
In the first inequality, we bounded the probability that $|\epsilon_{j,k,e}|\geq t$ by the probability that a standard normal random variable is larger than $t$ in absolute value. This is justified by the fact that $\epsilon_{j,k,e}$ has variance smaller than one for all indices.
\end{proof}

We begin with an auxiliary result for the proof of Theorem~\ref{results:Main_thm}, which is essentially a regularity result for $\hat{f}_n$ conditionally on the event $\mathcal{A}_n$ in~\eqref{sketch:GoodEvent}. In the following proofs, $C>0$ denotes a generic constant that may change from line to line.

\begin{proposition}\label{proofs:AuxLemma}
Let $\{\psi_{j,k,e}\}$ and $\{u_{j,k,e}\}$ denote the wavelet and vaguelette systems from Assumption~\ref{ass:OP}. For $n\geq e^L$, let $\hat{f}_n$ denote the estimator~\eqref{result:EST2} with parameter $\gamma_n$ given by~\eqref{IP:UnivGamma}. Then conditionally on the event $\mathcal{A}_n$ in~\eqref{sketch:GoodEvent} we have
\begin{align*}
(i) & \ \ \|\hat{f}_n-f\|_{B^{-d/2-\beta}_{\infty,\infty}}\leq C\, \gamma_n+C\frac{\|f\|_{L^{\infty}}+\log n}{\sqrt{n}},
  \\
  (ii) & \ \ \|\hat{f}_n-f\|_{BV}\leq \|f\|_{L^{\infty}}+2|f|_{BV}+\log n,
\end{align*}
for any $f\in BV\cap L^{\infty}(\Rd)$ with supp $f\subseteq[0,1]^d$, and a constant $C>0$ independent of $n$, $f$ and $\hat{f}_n$.
\end{proposition}

\begin{proof}
 For part $(i)$, the definition of the Besov $B^{-d/2-\beta}_{\infty,\infty}$ norm in terms of wavelet coefficients (see Section~\ref{sect:Notation}) yields
 \begin{align*}
 \|\hat{f}_n-f\|_{B^{-d/2-\beta}_{\infty,\infty}}&=\max_{(j,k,e)\in\Omega}2^{-\beta j}|\langle \psi_{j,k,e},\hat{f}_n-f\rangle|
 \\
 &\leq \max_{(j,k,e)\in\Omega_n}2^{-\beta j}|\langle \psi_{j,k,e},\hat{f}_n-f\rangle|+\max_{(j,k,e)\notin\Omega_n}2^{-\beta j}|\langle \psi_{j,k,e},\hat{f}_n-f\rangle|
 \\
 &\leq\max_{(j,k,e)\in\Omega_n}2^{-\beta j}|\kappa_j^{-1}\langle T^* u_{j,k,e},\hat{f}_n-f\rangle|+C\max_{(j,k,e)\notin\Omega_n}2^{-\beta j}\, \|\psi_{j,k,e}\|_{L^1}\|\hat{f}_n-f\|_{L^{\infty}}
 \\
 &\leq \max_{(j,k,e)\in\Omega_n}|\langle u_{j,k,e}, T\hat{f}_n-Tf\rangle|+C\frac{\|\hat{f}_n-f\|_{L^{\infty}}}{\sqrt{n}},
 \end{align*}
 where we used that $\kappa_j=2^{-j\beta}$ and that $\|\psi_{j,k,e}\|_{L^1}\leq C\, 2^{-jd/2}\|\psi_{j,k,e}\|_{L^2}$ for Daubechies wavelets (which are supported on a compact set). The numerator in the second term can be bounded by $\|f\|_{L^{\infty}}+\log n$ by construction of $\hat{f}_n$, while 
the first term can be bounded as
\begin{align*}
 \max_{(j,k,e)\in\Omega_n}\big|\langle &u_{j,k,e},T\hat{f}_n-Tf\rangle\big|
 \\
 &\leq \underbrace{\max_{(j,k,e)\in\Omega_n}\big|\langle u_{j,k,e},T\hat{f}_n\rangle-\langle u_{j,k,e},dY\rangle\big|}_{\leq\gamma_n}+\max_{(j,k,e)\in\Omega_n}\big|\langle u_{j,k,e},Tf\rangle- \langle u_{j,k,e},dY\rangle\big|
 \\
 &\leq \gamma_n+\max_{(j,k,e)\in\Omega_n}\frac{\sigma}{\sqrt{n}}\bigg|\int_{\M}u_{j,k,e}(x)\, dW(x)\bigg|\leq 2\gamma_n
\end{align*}
conditionally on $\mathcal{A}_n$, where in the second inequality we used the definition of $\hat{f}_n$. This completes the proof of $(i)$. 
The proof of $(ii)$ is analogous to the proof of Proposition 4 in~\cite{mTV}, so we do not reproduce it here.
\end{proof}

\begin{proof}[Proof of part a) of Theorem~\ref{results:Main_thm}]
We prove the claim of part a) of Theorem~\ref{results:Main_thm} conditionally on the event $\mathcal{A}_n$ in~\eqref{sketch:GoodEvent}, which by Proposition~\ref{app:GoodEvent} happens with probability $\mathbb{P}(\mathcal{A}_n)\geq 1-(\#\Omega_n)^{1-\kappa^2}$.
\\
Consider first the case $d\geq 2$, which gives $q^*:=1+2/(d+2\beta)\leq 2$. In this case, Proposition~\ref{ParticIntd=1} gives the interpolation 
inequality
 \begin{equation}
   \|\hat{f}_{n}-f\|_{L^{q}}\leq C\|\hat{f}_{n}-f\|_{\besov}^{\frac{2}{d+2\beta+2}}\|\hat{f}_{n}-f\|_{BV}^{\frac{d+2\beta}{d+2\beta+2}}\label{pf:middle1}
 \end{equation}
 for $q\leq 1+2/(d+2\beta)$. 
 Conditionally on $\mathcal{A}_n$ and for $n\geq e^L$, Proposition~\ref{proofs:AuxLemma} gives bounds for the terms in the right-hand side of~\eqref{pf:middle1}, and putting the last three equations together then yields
\begin{align*}
 \|\hat{f}_{n}-f\|_{L^q}&\leq C \bigg(\gamma_n+C\frac{\|f\|_{L^{\infty}}+\log n}{\sqrt{n}}\bigg)^{\frac{2}{d+2\beta+2}}\big(\|f\|_{L^{\infty}}+2|f|_{BV}+\log n\big)^{\frac{d+2\beta}{d+2\beta+2}}
 \\
 &\leq Cn^{-\frac{1}{d+2\beta+2}}\big(\sqrt{\log \#\Omega_n}+L+\log n\big)^{\frac{2}{d+2\beta+2}}\big(L+\log n\big)^{\frac{d+2\beta}{d+2\beta+2}}
 \\
 &\leq C\, n^{-\frac{1}{d+2\beta+2}}\, \log n
\end{align*}
using that $f\in BV_L$. Since $\#\Omega_n$ grows linearly in $n$ (recall Section~\ref{sect:Notation}), the claim follows.

For the case when $d=1$ and $\beta\geq 1/2$, we have $q^*\leq 2$ and the argument goes through as above.

Finally, the case $d=1$ and $\beta<1/2$ requires a special treatment, since then $q^*>2$. We use part b) of Proposition~\ref{ParticIntd=1}, which gives
 \begin{equation}
  \|\hat{f}_{n}-f\|_{L^q}\leq C (\log n )\, \|\hat{f}_{n}-f\|_{B^{-d/2-\beta}_{\infty,\infty}}^{\frac{2}{d+2\beta+2}}\|\hat{f}_{n}-f\|_{BV}^{\frac{d+2\beta}{d+2\beta+2}}+C\, n^{-1}\, \|\hat{f}_{n}-f\|_{L^{\infty}}^{\frac{2}{d+2\beta+2}}\, \|\hat{f}_{n}-f\|_{BV}^{\frac{d+2\beta}{d+2\beta+2}}
 \end{equation}
 for a constant $C>0$ and any $q\leq q^*$. 
 Conditionally on $\mathcal{A}_n$, we bound the terms in the right-hand side by Proposition~\ref{proofs:AuxLemma}, which for $n\geq e^L$ yields
 \begin{equation*}
 \|\hat{f}_n-f\|_{L^q}\leq C\, n^{-\frac{1}{d+2\beta+2}}\, (\log n)^2 +C\, n^{-1}\, \log n,
 \end{equation*}
 which gives the claim.
 
 We have proved the claim for the $L^q$-risk with $q\leq q^*:=1+2/(d+2\beta)$. For larger $q$, we use H\"older's inequality between the $L^{1+2/(d+2\beta)}$ and the $L^{\infty}$-risk, which gives the bound
 \begin{equation*}
 \|\hat{f}_n-f\|_{L^q}\leq \|\hat{f}_n-f\|_{L^{1+2/(d+2\beta)}}^{\frac{d+2\beta+2}{q(d+2\beta)}} \|\hat{f}_n-f\|_{L^{\infty}}^{1-\frac{d+2\beta+2}{q(d+2\beta)}}\leq C\, n^{-\frac{1}{q(d+2\beta)}}\,  (\log n)^{3-\min\{d,2\}}
 \end{equation*}
 for $q\geq 1+2/(d+2\beta)$. This completes the proof.
\end{proof}

\begin{proof}[Proof of part b) of Theorem~\ref{results:Main_thm}]
It follows from the convergence conditionally on $\mathcal{A}_n$ proved in part a) of the theorem. We omit the proof, as it is analogous to the proof of part b) of Theorem 1 in~\cite{mTV}.
 \end{proof}

\subsubsection{Proof of Theorem~\ref{results:Thm2}}

Here we prove Theorem~\ref{results:Thm2} by showing that the minimax rate over the smaller set
\begin{equation*}
(B^1_{1,1}\cap L^{\infty})_L:=\{g\in B^1_{1,1}\, |\, \|g\|_{L^{\infty}}\leq L, \ \|g\|_{B^1_{1,1}}\leq L, \ \textup{supp } g\subseteq[0,1]^d\}\subset BV_L
\end{equation*} 
with respect to the $L^q$-risk, $q\in[1,\infty)$, is not faster than $n^{-\vartheta_{q,\beta}}$. 
The proof of this is well-known in the dense case $q<1+2/(d+2\beta)$, where $\vartheta_{q,\beta}=\frac{1}{d+2\beta+2}$: it can be found e.g.~in Chapter 10 of~\cite{hardle2012wavelets} for $d=1$ and $T=id$, so we do not reproduce it here. Indeed, the generalization from $d=1$ to $d\geq 2$ is trivial. Concerning the difference between $T=id$ and general $T$, we show below how to adapt the construction of the alternatives in the case $q\geq 1+2/(d+2\beta)$, which indicates how to proceed in the dense regime (see e.g.~Theorem 3 in~\cite{cavalier2011inverse} for a different strategy for computing the minimax risk in inverse problems for the $L^2$-risk).

On the other hand, we have not found a lower bound in the literature for the regime $q\geq 1+2/(d+2\beta)$: only the construction in~\cite{mTV} for the particular case $\beta=0$ deals with that regime. Here we modify that proof and give a lower bound for general $\beta\geq 0$.

\begin{proof}[Proof of Theorem~\ref{results:Thm2}]

The proof follows the proof of Theorem 2 in~\cite{mTV} closely.
\\
\textbf{Construction of alternatives: } In the proof of Theorem 2 in~\cite{mTV}, a set of alternatives $\mathcal{G}:=\{g^{\epsilon}\, |\, \epsilon\in \{-1,+1\}^{S_j}\}$ is constructed such that
\begin{equation*}
g^{\epsilon}:=\gamma\sum_{(k,e)\in R_j}\epsilon_{k,e}\psi_{j,k,e},
\end{equation*}
where $\gamma\asymp 2^{-jd/2}$ is the signal strength, $\psi_{j,k,e}$ are orthonormal Daubechies wavelets, and $(k,e)\in R_j\subseteq \{0,\ldots,2^j-1\}^d\times E_j$, $E_j=\{0,1\}^d\backslash\{0\}$, are indices such that $S_j=\# R_j=2^{j(d-1)}$. These functions are chosen to satisfy $\|g^{\epsilon}\|_{B^{1}_{1,1}}\leq L$, $\|g^{\epsilon}\|_{L^{\infty}}\leq L$ and
\begin{align}
\delta:=\inf_{\epsilon\neq\epsilon'}\|g^{\epsilon}-g^{\epsilon'}\|_{L^q}=2\|\gamma\psi_{j,k,e}\|_{L^q}=2\gamma\, 2^{jd(\frac{1}{2}-\frac{1}{q})}\, \|\psi\|_{L^q}\asymp 2^{-jd/q} \label{Assouaddelta}.
\end{align}
\textbf{Lower bound: } We use now Assouad's lemma for lower bounding the $L^q$-risk over $(B^1_{1,1}\cap L^{\infty})_L$. We reproduce the claim (Lemma 10.2 in~\cite{hardle2012wavelets}) for completeness. 
\begin{lemma}\label{lemma:Assouad}
For $\epsilon\in\{-1,+1\}^{S_j}$ and $(k,e)\in R_j$, define $\epsilon_{*k,e}:=(\epsilon_{(k_1,e_1)}',\ldots,\epsilon_{(k_{S_j},e_{S_j})}')$, where
\begin{equation*}
\epsilon_{(k'e')}'=\begin{cases}
\epsilon_{(k,e)}\ & \textup{ if } (k',e')\neq (k,e),
\\
-\epsilon_{(k,e)}\ & \textup{ if } (k',e')= (k,e).
\end{cases}
\end{equation*}
Assume there exist constants $\lambda,p_0>0$ such that
\begin{equation}
\mathbb{P}_{Tg^{\epsilon}}\big(LR(Tg^{\epsilon_{*k,e}},Tg^{\epsilon})>e^{-\lambda}\big)\geq p_0, \ \ \forall\epsilon, \ \forall n,\label{eq:lemmaAssouad}
\end{equation}
where $\mathbb{P}_{Tg^{\epsilon}}$ denotes the probability with respect to observations drawn from $Tg^{\epsilon}$ in the white noise model~\eqref{results:MODEL}, and $LR(Tg^{\epsilon_{*k,e}},Tg^{\epsilon})$ denotes the likelihood ratio between the observations associated to $Tg^{\epsilon_{*k,e}}$ and $Tg^{\epsilon}$. Then any estimator $\hat{f}$ based on observations~\eqref{results:MODEL} satisfies
\begin{equation*}
\sup_{g^{\epsilon}\in \mathcal{G}}\mathbb{E}_{Tg^{\epsilon}}\|\hat{f}-g^{\epsilon}\|_{L^q}\geq \frac{e^{-\lambda}\, p_0}{2}\, \delta\, S_j^{1/q},
\end{equation*}
where $\delta$ is defined in~\eqref{Assouaddelta}.
\end{lemma}

\textbf{Verification of~\eqref{eq:lemmaAssouad}: }
With the same argument as the proof of Theorem 2 in~\cite{mTV}, condition~\eqref{eq:lemmaAssouad} holds provided that the Kullback-Leibler divergence between observations from two alternatives satisfies $K(dP_{Tg^{\epsilon_{*k,e}}},dP_{Tg^{\epsilon}})\leq c$ for a small enough constant $c>0$. A standard computation gives
\begin{equation*}
K(dP_{Tg^{\epsilon_{*k,e}}},dP_{Tg^{\epsilon}})=\frac{n}{2\sigma^2}\|Tg^{\epsilon_{*k,e}}-Tg^{\epsilon}\|_{L^2}^2=\frac{n\gamma^2}{2\sigma^2}\|T\psi_{j,k,e}\|_{L^2}^2\leq \frac{n\gamma^2\, 2^{-2j\beta}}{2\sigma^2}
\end{equation*}
using~\eqref{lowBound:OP}, so choosing $\gamma^2\asymp 2^{-jd}\asymp n^{-\frac{d}{d+2\beta}}$ gives~\eqref{eq:lemmaAssouad}. 
\\

\textbf{Application of Lemma~\ref{lemma:Assouad}: } The conclusion of the lemma applies, and we can lower bound the $L^q$-risk over the class $(B^1_{1,1}\cap L^{\infty})_L$ by the risk over $\mathcal{G}$, i.e.,
\begin{equation}
\sup_{f\in (B^1_{1,1}\cap L^{\infty})_L}\mathbb{E}_{Tf}\|\hat{f}-f\|_{L^q}\geq \sup_{g^{\epsilon}\in \mathcal{G}}\mathbb{E}_{Tg^{\epsilon}}\|\hat{f}-g^{\epsilon}\|_{L^q}\geq \frac{e^{-\lambda}\, p_0}{2}\, \delta\, 2^{j\Delta/q}\label{LowRisk1}
\end{equation}
for any estimator $\hat{f}$. Choosing as above $2^{j}\asymp n^{1/(d+2\beta)}$, the definition~\eqref{Assouaddelta} for $\delta$ gives the bound
\begin{align*}
\sup_{f\in (B^1_{1,1}\cap L^{\infty})_L}\mathbb{E}_{Tf}\|\hat{f}-f\|_{L^q}\geq c\, \delta\, 2^{j\Delta/q}\asymp  2^{-j(d-\Delta)/q} \asymp  n^{-\frac{1}{q(d+2\beta)}},
\end{align*}
which completes the proof.
\end{proof}

\subsection{Examples}\label{sect:Examp}

\subsubsection{Radon transform}\label{sect:ExampRad}

Due to its application in nondestructive imaging, in particular in medical applications, tomography is a very relevant inverse problem. While there are plenty of mathematical models for tomography, which mainly depend on the type of tomography and the geometry of the detector (see e.g.~Chapter 1 in~\cite{scherzer2009variational}), in this section we will exemplarily consider tomography modeled by the Radon transform. For simplicity we consider here the two dimensional case, in which the Radon transform of a function $g$ is given by its line integrals along different directions, see~\eqref{examples:Radon}.

Functions in the range of $T$ are supported on cylindrical sets of the form $\mathbb{M}=\mathbb{R}\times[0,2\pi)$. Moreover, the domain of $T$ consists of functions $g\in L^2(\Rd)$ whose Fourier transform satisfies $|\xi|^{-1/2}\mathcal{F}[g](\xi)\in L^2$, see~\cite{donoho1995nonlinear}. This is a condition on the low frequencies which essentially ensures that local averages remain reasonably small.

In this section we will show how to apply the estimation framework developed above to this type of inverse problems. For that, let $\{\psi_{j,k,e}\}$ denote a basis of Daubechies wavelets as described in Section~\ref{sect:Notation}. For $(j,k,e)\in\Omega$, define the vaguelettes by
\begin{equation}
u_{j,k,e}(r,\theta)=\frac{2^{-j/2}}{(2\pi)^2}\int_{\mathbb{R}}|\rho|\, \mathcal{F}[\psi_{j,k,e}](\rho\cos\theta,\rho\sin\theta)\, e^{ir\rho}\, d\rho.
\end{equation}
It is easy to verify directly (see e.g.~Chapter 2 in~\cite{natterer1986mathematics}) that the vaguelettes satisfy the equation
\begin{equation*}
T^*u_{j,k,e}=\kappa_j\, \psi_{j,k,e}
\end{equation*}
for generalized critical values $\kappa_j=2^{-j/2}$. Moreover,
\begin{equation*}
c_1\leq \|u_{j,k,e}\|_{L^2}\leq c_2 \ \ \ \forall (j,k,e)\in\Lambda,
\end{equation*}
for constants $c_1,c_2$ depending on $\psi_{0,0,e}$, see Section 3.3 in~\cite{donoho1995nonlinear} for a proof of this claim. Let us remark that the system $\{u_{j,k,e}\}$ is part of a WVD for $T$ (see~\cite{donoho1995nonlinear} for the details).

Altogether, the observations above imply that the Radon transform satisfies Assumption~\ref{ass:OP} with $\beta=1/2$ in dimension $d=2$. By Theorem~\ref{results:Thm2}, the multiscale total variation estimator~\eqref{result:EST2} is nearly minimax optimal for recovering a function $f\in BV_L$ from noisy Radon observations. We remark that the same analysis can be performed for the Radon transform in higher dimensions, in which case $\beta=(d-1)/2$, for the X-ray transform, with $\beta=1/2$ for any dimension~\citep{natterer1986mathematics}, as well as for other tomography operators, such as photoacoustic and thermoacoustic tomography (see e.g.~\cite{haltmeier2013inversion}).

\subsubsection{Convolution}\label{sect:ExampConv}

Let $T$ denote the convolution operator with a kernel $K\in L^1(\Rd)$, i.e.,
\begin{equation*}
Tg(x):=\int_{\Rd} K(x-y)g(y)\, dy.
\end{equation*}
We let $\M=\Rd$, and by Young's inequality $T$ is a bounded operator from $\mathcal{D}(T)=L^2(\Rd)$ to itself whose operator norm equals $\|K\|_{L^1}$. 
The inverse problem~\eqref{results:MODEL} with a convolution operator $T$ is a model for a myriad of applications in image and signal processing, including microscopy and astronomy models (see e.g.~\cite{bertero2009image}). The problem of recovering a signal $f$ from noisy measurement of its convolution $Tf$ is hence of extreme practical relevance. In this section we show that the multiscale TV-estimator~\eqref{result:EST2} solves this problem in a minimax optimal sense. 

For that, we need to impose regularity conditions on $T$, which naturally have the form of a decay condition on the Fourier transform of $K$. In particular, we assume that the kernel $K$ satisfies 
\begin{equation}
a_1\, (1+|\xi|^2)^{-\beta/2}\leq |\mathcal{F}[K](\xi)|\leq a_2\, (1+|\xi|^2)^{-\beta/2} \ \ \forall\xi\in\mathbb{R}^d\label{IP:kernelDecay}
\end{equation}
for constants $a_1,a_2\geq 0$ and some $\beta\geq 0$. Given a basis of Daubechies wavelets $\{\psi_{j,k,e}\}$ like that in Section~\ref{sect:Notation} with $\min\{R,D\}>\max\{1,d/2+\beta\}$, define the system of functions
\begin{equation}
u_{j,k,e}(x):=2^{j(d/2-\beta)}\mathcal{F}^{-1}\bigg[\frac{\mathcal{F}[\psi_{0,0,e}](\cdot)}{\mathcal{F}[K](-2^j\cdot)}\bigg]\big(2^jx-k\big)\label{IP:DeconvDiction}
\end{equation}
indexed by the set $\Omega$ in~\eqref{results:omega}. These functions satisfy the following relations
\begin{align*}
& T^*u_{j,k,e}=\kappa_j\, \psi_{j,k,e} \ \ \textup{ where } \kappa_j=2^{-j\beta},
\\
& c_1\leq \|u_{j,k,e}\|_{L^2}\leq c_2 \ \ \forall (j,k,e)\in\Omega,
\end{align*}
where we can choose $c_1=\min_{e\in \{0,1\}^d}\|(-\Delta)^{\beta/2}\psi_{0,0,e}\|_{L^2}$ and $c_2=\max_{e\in\{0,1\}^d}\|\psi_{0,0,e}\|_{H^{\beta}}$ (see Proposition~\ref{aux:Convolution} for the proof). These results show that the convolution operator $T$ under the assumptions above satisfies Assumption~\ref{ass:OP}. By Theorem~\ref{results:Thm2} we conclude that the multiscale TV-estimator is minimax optimal for estimating functions $f\in BV_L$, up to logarithmic factors.

\subsection{Nonparametric inverse regression model}\label{Rem_Discr}

So far we have discussed the estimator $\hat{f}_n$ based on observations from the white noise model~\eqref{results:MODEL}. In practice, however, one naturally has access to discretely sampled data, which makes it more realistic to model the observations with the nonparametric regression model
\begin{equation}
Y_i=Tf(x_i)+\sigma\, \epsilon_i, \ \ \ x_i\in\Gamma_n, \ \ i=1,\ldots,n.\label{Res:RegM}
\end{equation}
 Here we assume that $n=m^d$ for some $m\in\mathbb{N}$, and that the design points belong to an equidistant grid
\begin{equation*}
 \Gamma_n:=\bigg\{\bigg(\frac{k_1}{m},\cdots,\frac{k_d}{m}\bigg)\, \bigg|\, k_i\in\{1,\ldots,m\},\ i=1,\ldots,d\bigg\}.\label{Gamman}
\end{equation*}
 Of course, different grids may be used, depending on the operator $T$ and the domain $\M$ under consideration. For simplicity of the analysis, we assume in this section that $\M=(0,1)^d$. This is the case when $T$ is the identity operator, a suitable convolution operator, or integration, to mention just a few examples. In~\eqref{Res:RegM}, $\epsilon_i$ are independent standard normal random variables, and $\sigma>0$ plays the role of the standard deviation of the noise.

Given observations~\eqref{Res:RegM}, our goal is to estimate the function $f$. We do so by discretizing our construction of the multiscale TV-estimator from Definition~\ref{results:definition}. 
Let $\{u_{\omega}^n \, |\, \omega\in\Omega_n\}$ be a dictionary of discretized vaguelettes, i.e., each $u_{\omega}^n$ is a vector of $n$ values
\begin{equation*}
\big(u_{\omega}^n\big)_i=n^{-1/2}\,u_{\omega}(x_i)\ \ \textup{ for } i=1,\ldots,n, \ \ x_i\in\Gamma_n,
\end{equation*}
which are the evaluations of the vaguelette $u_{\omega}$ at the grid points. The scaling factor $n^{-1/2}$ is chosen so that
\begin{equation*}
\sum_{x_i\in\Gamma_n}\big|\big(u_{\omega}^n\big)_i \big|^2\rightarrow \|u_{\omega}\|_{L^2}^2=1 \ \ \ \textup{  as } \ \ n\rightarrow\infty,
\end{equation*}
for any $\omega\in\Omega_n$, i.e., so that the vectors $u_{\omega}^n$ have asymptotically unit norm in an $L^2$ sense.

In this setting, the multiscale TV-estimator takes the form
\begin{equation}
 \hat{f}_{D}\in\underset{g\in \mathcal{F}_n}{\textup{ argmin }}|g|_{BV}\ \textup{ subject to } \ \max_{\omega\in\Omega_n}\big|\sum_{x_i\in\Gamma_n}\big(u_{\omega}^n\big)_i\big(Tg(x_i)-Y_i\big)\big|\leq\kappa\, c_2\, \sigma\, \sqrt{2\log\#\Omega_n},\label{Disc_Est}
\end{equation}
where $c_2>0$ is the upper frame constant for the continuous vaguelettes in Assumption~\ref{ass:OP}.

We can now analyze the estimator~\eqref{Disc_Est} following the same strategy as we did in the white noise model. The only difference will be that, above, we related the constraint on the vaguelette coefficients to the Besov $B^{-d/2-\beta}_{\infty,\infty}$ norm. Since here we only have access to the \textit{discretized} vaguelette coefficients, there is an additional discretization error caused by the approximation of the vaguelette coefficients by their discretized counterparts. That error is given by
\begin{equation}
\delta_n:=\sup_{g\in\mathcal{F}_n}\bigg|\int_{[0,1]^d}u_{\omega}(x)Tg(x)\, dx- n^{-1}\sum_{i=1}^nu_{\omega}(x_i)Tg(x_i)\bigg|.\label{discr_error}
\end{equation}
Proceeding as in the proof of Proposition~\ref{proofs:AuxLemma}, we see that $\hat{f}_D$ satisfies the error bounds
\begin{align*}
(i) & \ \ \|\hat{f}_D-f\|_{B^{-d/2-\beta}_{\infty,\infty}}\leq C\, \gamma_n+C\frac{\|f\|_{L^{\infty}}+\max\{L,\log n\}}{\sqrt{n}}+2\delta_n,
  \\
  (ii) & \ \ \|\hat{f}_D-f\|_{BV}\leq \|f\|_{L^{\infty}}+2|f|_{BV}+\max\{L,\log n\},
\end{align*}
conditionally on the event $\mathcal{A}_n$ in~\eqref{sketch:GoodEvent}. Following the proof of Theorem~\ref{results:Main_thm}, we get the result
\begin{equation*}
    \sup_{f\in BV_L}\mathbb{E}\big[\|\hat{f}_{D}-f\|_{L^q}\big]\leq C\, \max\{n^{-1/2},\delta_n\}^{2\vartheta_{q,\beta}}\, (\log n)^{3-\min\{d,2\}}
 \end{equation*}
for $q\in[1,\infty)$ and $n$ large enough. Here we have the following trade-off: if $\delta_n$ is of smaller order that $n^{-1/2}$, then $\hat{f}_D$ attains the same rate as the multiscale TV-estimator based on observations from the white noise model. On the other hand, if $\delta_n$ is of bigger order than $n^{-1/2}$, the discretization error dominates and $\hat{f}_D$ performs worse than $\hat{f}_n$. The different performance of the multiscale TV-estimator in the white noise and in the nonparametric regression models hence boils down to a purely approximation theoretic question. 

It remains now to bound the discretization error $\delta_n$. For that, notice that it is entirely determined by the smoothness of $u_{\omega}\, Tg$. Recall that $g\in BV\cap L^{\infty}$ and that $T$ is a smoothing operator. 
Consider the following examples.
\begin{itemize}
\item[1)] Let $T$ be the identity operator. Then $u_{\omega}=\psi_{\omega}$ is a smooth wavelet basis, and $Tg\in BV\cap L^{\infty}$. Consequently, the product $u_{\omega}\, Tg$ is at most a function of bounded variation, for which we have $\delta_n=O(n^{-1/d})$ (see e.g.~Chapter 5 in~\cite{Evans}). In this case, the discretization error is of lower order for $d=1,2$, while it dominates for $d\geq 3$. 
\item[2)] In particular cases, for $T=id$, the error in $d\geq 3$ can be improved. For instance, if $g$ is a piecewise constant function and if $u_{\omega}$ is smooth enough and has vanishing moments. In that case, the discretization error can be of smaller order due to the vanishing moments of $\psi_{\omega}$. We do not pursue this idea further.
\item[3)] If $T$ is a convolution operator as in Section~\ref{sect:ExampConv}, then by Fourier inversion we can show that $u_{\omega}$ is continuous. Moreover, if the kernel decays fast enough, $Tg$ will be a continuous function as well, and so will be $u_{\omega}\, Tg$. Hence we have the same discretization error $\delta_n=O(n^{-1/d})$ as above. There is nevertheless an important caveat here: as opposed to wavelets, vaguelettes do not have in general vanishing moments. Consequently, this error cannot be improved by assuming that $Tg$ is e.g.~piecewise constant.
\end{itemize}

We have argued that the difference between the multiscale TV-estimator in the white noise and the nonparametric inverse regression models arises from a discretization error. In particular, the error appears in the convergence rate in the nonparametric regression model, eventually making it slower. Importantly, for $d=2$ the error behaves as $\delta_n=O(n^{-1/2})$, and so the multiscale estimator attains the optimal convergence rate $n^{-\vartheta_{q,\beta}}$ for imaging problems in the discretized model~\eqref{Res:RegM}. 

 More generally, the difference between the white noise and the nonparametric inverse problem models can be measured with the theory of asymptotic equivalence. While that theory is well understood when $T=id$ (\cite{brown1996},~\cite{reiss2008}), there are considerably fewer results for general operators $T$ (see~\cite{grama1998asymptotic} and~\cite{meister2011asymptotic}). In particular,~\cite{meister2011asymptotic} proves asymptotic equivalence in a functional linear regression model provided that the unknown function is suitably smooth, which is reminiscent of our analysis above to control $\delta_n$ based on the smoothness of $Tg$.

\section{Auxiliary analytical results}\label{sect:Proofs}


For simplicity, we prove the two parts of Proposition~\ref{ParticIntd=1} separately. 
They rely on an interpolation inequality proved by~\cite{cohen2003harmonic}, which we reproduce here.
\begin{theorem}[Theorem 1.5 in~\cite{cohen2003harmonic}]\label{thm:Cohen}
Let $s\in\mathbb{R}$ and $1<p\leq\infty$, and assume that $\gamma:=1+(s-1)p'/d$ satisfies either $\gamma>1$ or $\gamma<1-1/d$, where $p'$ denotes the H\"older conjugate of $p$. Then for any $0<\theta<1$ such that
\begin{equation*}
\frac{1}{q}=\frac{1-\theta}{p}+\theta, \ \ t=(1-\theta)s+\theta
\end{equation*}
we have the inequality
\begin{equation}
\|g\|_{B^{t}_{q,q}}\leq C\, \|g\|_{B^{s}_{p,p}}^{1-\theta}\|g\|_{BV}^{\theta}\label{sketch:Cohen}
\end{equation}
for any function $g\in BV\cap B^{s}_{p,p}$ and a constant $C>0$ depending on $p,q,s$ and $d$ only.
\end{theorem}

\begin{proof}[Proof of part a) of Proposition~\ref{ParticIntd=1}]
First, Theorem~\ref{thm:Cohen} with $s=-d/2-\beta$ and $p=\infty$ gives
\begin{equation*}
 \|g\|_{B^0_{q^*,q^*}}\leq C\, \|g\|_{B^{-d/2-\beta}_{\infty,\infty}}^{\frac{2}{d+2\beta+2}}\|g\|_{BV}^{\frac{d+2\beta}{d+2\beta+2}}
\end{equation*}
for any smooth enough $g$. It remains to show that the $L^q$-norm, $q\in[1,q^*]$, can be upper bounded by the $B^0_{q^*,q^*}$-norm. But that is indeed the case, due to the continuous embedding
\begin{equation}
B_{r,r}^0(\mathbb{R}^d)\hookrightarrow L^r(\mathbb{R}^d) \label{Embed1}
\end{equation}
for $r\in(1,2]$. Indeed, continuity of the embedding follows from Proposition 2 in Section 2.3.2 in~\cite{TriebelI}. It states that, for $0<q\leq\infty$, $0<p<\infty$ and $s\in\mathbb{R}$, the embedding
\begin{equation*}
B^s_{p,\min\{p,q\}}(\Rd)\hookrightarrow F^s_{p,q}(\Rd)
\end{equation*}
is continuous. Moreover, equation (2) in Section 2.3.5 in~\cite{TriebelI} states that
\begin{equation*}
F^{0}_{p,2}(\Rd)=L^p(\Rd)
\end{equation*}
for $p\in(1,\infty)$.
These two facts imply that
\begin{equation*}
B^0_{r,r}(\Rd)=B^0_{r,\min\{r,2\}}(\Rd)\hookrightarrow F_{r,2}^0(\Rd)=L^r(\Rd) \ \ \ \forall r\in(1,2],
\end{equation*}
which completes the proof of~\eqref{Embed1}. The extension to the $L^1$-risk follows by compact support.
\end{proof}

The proof of part b) of Proposition~\ref{ParticIntd=1} relies on the following result.

\begin{proposition}\label{new_d=1}
Let $g\in L^{\infty}\cap BV$ satisfy supp $g\subseteq[0,1]^d$, and let $q\in[2,3]$. Then for any $J\in\mathbb{N}$ we have
\begin{equation*}
\|g\|_{L^{q}}\leq C\,  J\, \|g\|_{B^0_{q,q}}+C\, 2^{-J/q}\|g\|_{L^{\infty}}^{1-1/q}\|g\|_{BV}^{1/q}
\end{equation*}
for a constant $C>0$ independent of $g$.
\end{proposition}

The proof of Proposition~\ref{new_d=1} uses the following lemma.

\begin{lemma}\label{lemma:wavelets}
Let $\{\psi_{j,k,e}\, |\, (j,k,e)\in\Omega\}$ denote a basis of compactly supported wavelets in $L^2(\Rd)$. For any $q\in[2,3]$ there is a constant $C_{\psi,q}$ such that
\begin{equation*}
\int_{\Rd}\bigg|\sum_{(k,e)\in P_j^d\times E_j}c_{j,k,e}\psi_{j,k,e}(x)\bigg|^q\, dx\leq C_{\psi,q}\, 2^{jqd(1/2-1/q)}\sum_{(k,e)\in P_j^d\times E_j}|c_{j,k,e}|^q
\end{equation*}
for any $j\in\mathbb{N}$ and any coefficients $\{c_{j,k,e}\}$, where
\begin{equation*}
P_j^d:=\{k\in \mathbb{Z}^d\, |\, (j,k,e)\in\Omega, \ \ \textup{ supp }\psi_{j,k,e}\cap (0,1)^d\neq \emptyset\}.
\end{equation*}
\end{lemma}

\begin{proof}[Proof of Lemma~\ref{lemma:wavelets}]
We prove the lemma by showing the extreme cases $q=2$ and $q=3$, and then applying the Riesz-Thorin interpolation theorem (see e.g.~\cite{stein2016introduction}) to the bounded operator
\begin{align*}
A_j:\, &\ell^q(P_j^d\times E_j)\rightarrow L^q(\Rd)
\\
&\{c_{j,k,e}\}_{(k,e)\in P_j^d\times E_j}\mapsto \sum_{(k,e)\in P_j^d\times E_j}c_{j,k,e}\psi_{j,k,e},
\end{align*}
which gives the claim for all $q\in[2,3]$. 
The claim for $q=2$ follows by the orthonormality of the wavelet basis. 
For $q=3$, the claim follows with the same argument as Lemma 2 in~\cite{mTV}: the only difference is that the functions there are defined on the torus $\mathbb{T}^d$, and here on the cube $[0,1]^d$. This completes the proof.
\end{proof}

\begin{proof}[Proof of Proposition~\ref{new_d=1}]
Let $\{\psi_{j,k,e}\}$ be a basis of compactly supported wavelets. Writing $g$ formally as its wavelet series we have for any $q\in[2,3]$
\begin{equation}
\|g\|_{L^q}=\bigg\|\sum_{j\in\mathbb{N}}\sum_{k,e}c_{j,k,e}\psi_{j,k,e}\bigg\|_{L^q}\leq \bigg\|\sum_{j\leq J}\sum_{k,e}c_{j,k,e}\psi_{j,k,e}\bigg\|_{L^q}+\bigg\|\sum_{j>J}\sum_{k,e}c_{j,k,e}\psi_{j,k,e}\bigg\|_{L^q}\label{aux:d1_first}
\end{equation}
for any $J\in\mathbb{N}$. Since supp $g\subseteq[0,1]^d$, the sums are over $(k,e)\in P_j^d\times E_j$. Using Lemma~\ref{lemma:wavelets}, the first term can be bounded as
\begin{align*}
\bigg\|\sum_{j\leq J}\sum_{k,e}c_{j,k,e}\psi_{j,k,e}\bigg\|_{L^q}&\leq \sum_{j\leq J}\bigg(C_{\psi,q}2^{jqd(1/2-1/q)}\sum_{(k,e)}|c_{j,k,e}|^q\bigg)^{1/q}
\\
&\leq C_{\psi,q}^{1/q}\, J\, \bigg(\max_{j\leq J}\, 2^{jqd(1/2-1/q)}\sum_{(k,e)}|c_{j,k,e}|^q\bigg)^{1/q}
\\
&\leq  C_{\psi,q}^{1/q}\, J\,\|g\|_{B^0_{q,q}},
\end{align*}
which gives the first term of the claim. For the second term, we use that $g\in L^{\infty}$ and $g\in BV$, which means that the wavelet coefficients of $g$ satisfy the bounds
\begin{equation*}
\max_{(k,e)\in P_j^d\times E_j}|c_{j,k,e}|\leq 2^{-jd/2}\, \|g\|_{L^{\infty}} \ \ \textup{ and }\ \   \sum_{(k,e)\in P_j^d\times E_j}|c_{j,k,e}|\leq 2^{j(d/2-1)}\, \|g\|_{BV},
\end{equation*}
for any $j\in\mathbb{N}$, where the first inequality follows from the compact support of the wavelets and H\"older's inequality, and the second follows from the embedding $BV\subset B^1_{1,\infty}$. 
Using Lemma~\ref{lemma:wavelets} and these bounds, the second term in~\eqref{aux:d1_first} can be bounded as
\begin{align*}
\bigg\|\sum_{j>J}\sum_{k,e}c_{j,k,e}\psi_{j,k,e}\bigg\|_{L^q}&\leq\sum_{j>J} \bigg(C_{\psi,q}2^{jqd(1/2-1/q)}\sum_{(k,e)\in P_j^d\times E_j}|c_{j,k,e}|^q\bigg)^{1/q}
\\
&\leq C_{\psi,q}^{1/q}\sum_{j>J} \bigg(2^{jqd(1/2-1/q)}\, 2^{-jd(q-1)/2}\, \|g\|_{L^{\infty}}^{q-1}2^{j(d/2-1)}\|g\|_{BV}\bigg)^{1/q}
\\
&\leq C_{\psi,q}^{1/q}\, \|g\|_{L^{\infty}}^{1-1/q}\, \|g\|_{BV}^{1/q}\sum_{j>J} 2^{-j/q},
\end{align*}
which gives the claim.
\end{proof}

\begin{proof}[Proof of part b) of Proposition~\ref{ParticIntd=1}]
Let $q^*:=1+2/(d+2\beta)$ and assume that $q^*>2$. Notice that $q^*\leq 3$ for $d\in\mathbb{N}$ and $\beta\geq 0$. The claim follows from Theorem~\ref{thm:Cohen} with $s=-d/2-\beta$ and $p=\infty$, which gives a bound on the $B^0_{q^*,q^*}$ norm. The $L^q$-norm, $q\in[1,q^*]$, can be upper bounded by the $L^{q^*}$-norm, which itself can be upper bounded by the $B^0_{q^*,q^*}$ norm using Proposition~\ref{new_d=1} below. Choosing $J=\lceil q^*\log n\rceil$ yields the claim.
\end{proof}

\begin{proposition}\label{aux:Convolution}
In the setting of Section~\ref{sect:ExampConv} we have
\begin{align*}
& T^*u_{j,k,e}=\kappa_j\, \psi_{j,k,e} \ \ \textup{ where } \kappa_j=2^{-j\beta},
\\
& c_1\leq \|u_{j,k,e}\|_{L^2}\leq c_2 \ \ \forall (j,k,e)\in\Omega,
\end{align*}
where we can choose $c_1=\min_{e\in \{0,1\}^d}\|(-\Delta)^{\beta/2}\psi_{0,0,e}\|_{L^2}$ and $c_2=\max_{e\in\{0,1\}^d}\|\psi_{0,0,e}\|_{H^{\beta}}$
\end{proposition}

\begin{proof}
Notice that the Fourier transform of the elements $u_{j,k,e}$ is given by
\begin{equation}
\mathcal{F}[u_{j,k,e}](\xi)=2^{-jd/2-j\beta}\, e^{-i\xi\cdot k 2^{-j}}\frac{\mathcal{F}[\psi_{0,0,e}](2^{-j}\xi)}{\mathcal{F}[K](-\xi)}.
\end{equation}
The first claim of the proposition follows trivially by construction of the $u_{j,k,e}$: we essentially use that $T^*$ acts by convolution with $K(-\cdot)$, which in Fourier domain is the product with $\mathcal{F}[K](-\cdot)$. For the bounds in the $L^2$ norm, we use Plancherel's theorem, i.e.
\begin{align}
\|u_{j,k,e}\|_{L^2}^2=\|\mathcal{F}[u_{j,k,e}]\|_{L^2}^2&=2^{-jd-2j\beta}\int_{\mathbb{R}^d}\bigg|\frac{\mathcal{F}[\psi_{0,0,e}](2^{-j}\xi)}{ \mathcal{F}[K](-\xi)}\bigg|^2\, \frac{d\xi}{(2\pi)^d}\nonumber
\\
&\asymp 2^{-jd-2j\beta}\int_{\mathbb{R}^d}\big(1+|\xi|^2\big)^{\beta}\big|\mathcal{F}[\psi_{0,0,e}](2^{-j}\xi)\big|^2\, d\xi\nonumber
\\
&= 2^{-2j\beta}\int_{\mathbb{R}^d}\big(1+|2^j\xi|^2\big)^{\beta}\big|\mathcal{F}[\psi_{0,0,e}](\xi)\big|^2\, d\xi,\label{Aux:Plancher}
\end{align}
where in the second line we used the bounds~\eqref{IP:kernelDecay} on the Fourier transform of the kernel $K$. The expression in the right-hand side can now be easily bounded from below as
\begin{align*}
2^{-2j\beta}\int_{\mathbb{R}^d}\big(1+|2^j\xi|^2\big)^{\beta}\big|\mathcal{F}[\psi_{0,0,e}](\xi)\big|^2\, d\xi&\geq 2^{-2j\beta}\int_{\mathbb{R}^d}|2^j\xi|^{2\beta}\big|\mathcal{F}[\psi_{0,0,e}](\xi)\big|^2\, d\xi
\\
&=\big\||\xi|^{\beta}\mathcal{F}[\psi_{0,0,e}]\big\|_{L^2}^2=\|(-\Delta)^{\beta/2}\psi_{0,0,e}\|_{L^2}^2,
\end{align*}
again by Plancherel's theorem. On the other hand, the right-hand side of~\eqref{Aux:Plancher} can be upper-bounded as
\begin{align*}
2^{-2j\beta}\int_{\mathbb{R}^d}\big(1+|2^j\xi|^2\big)^{\beta}\big|\mathcal{F}[\psi_{0,0,e}](\xi)\big|^2\, d\xi& \leq 2^{-2j\beta}\int_{\mathbb{R}^d}\big(2^{2j}+|2^j\xi|^2\big)^{\beta}\big|\mathcal{F}[\psi_{0,0,e}](\xi)\big|^2\, d\xi
\\
&=\big\|\big(1+|\xi|^2\big)^{\beta/2}\mathcal{F}[\psi_{0,0,e}]\big\|_{L^2}^2=\|\psi_{0,0,e}\|_{H^{\beta}}^2.
\end{align*}
This yields the claim.
\end{proof}

\section*{Funding}

This work was supported by the Deutsche Forschungsgemeinschaft [RTG 2088-B2 to M.A., CRC 755-A4 to A.M.].

\section*{Acknowledgment}

The authors thank Dr.~Housen Li and Dr.~Frank Werner for helpful discussions.



\begin{thebibliography}{}

\bibitem[Abramovich and Silverman, 1998]{abramovich1998wavelet}
Abramovich, F.~U. and Silverman, B.~W. (1998).
\newblock Wavelet decomposition approaches to statistical inverse problems.
\newblock {\em Biometrika}, 85(1):115--129.

\bibitem[Assouad, 1983]{assouad1983deux}
Assouad, P. (1983).
\newblock Deux remarques sur l'estimation.
\newblock {\em C. R. Math. Acad. Sci. Paris}, 296(23):1021--1024.

\bibitem[Bertero et~al., 2009]{bertero2009image}
Bertero, M., Boccacci, P., Desider{\`a}, G., and Vicidomini, G. (2009).
\newblock Image deblurring with \textup{Poisson} data: from cells to galaxies.
\newblock {\em Inverse Problems}, 25(12):123006.

\bibitem[Brown and Low, 1996]{brown1996}
Brown, L.~D. and Low, M.~G. (1996).
\newblock Asymptotic equivalence of nonparametric regression and white noise.
\newblock {\em Ann. Statist.}, 24(6):2384--2398.

\bibitem[Cand{\`e}s and Donoho, 2002]{candes2002recovering}
Cand{\`e}s, E.~J. and Donoho, D.~L. (2002).
\newblock Recovering edges in ill-posed inverse problems: Optimality of
  curvelet frames.
\newblock {\em Ann. Statist.}, 30(3):784--842.

\bibitem[Cand{\`e}s and Guo, 2002]{CandesGuo}
Cand{\`e}s, E.~J. and Guo, F. (2002).
\newblock New multiscale transforms, minimum total variation synthesis:
  Applications to edge-preserving image reconstruction.
\newblock {\em Signal Processing}, 82(11):1519--1543.

\bibitem[Cavalier, 2011]{cavalier2011inverse}
Cavalier, L. (2011).
\newblock Inverse problems in statistics.
\newblock In {\em Inverse problems and high-dimensional estimation}, pages
  3--96. Springer Berlin Heidelberg.

\bibitem[Chambolle and Pock, 2011]{chambolle}
Chambolle, A. and Pock, T. (2011).
\newblock A first-order primal-dual algorithm for convex problems with
  applications to imaging.
\newblock {\em J. Math. Imaging Vision}, 40(1):120--145.

\bibitem[Clason et~al., 2010]{clason}
Clason, C., Jin, B., and Kunisch, K. (2010).
\newblock A semismooth newton method for \textup{$L^1$} data fitting with
  automatic choice of regularization parameters and noise calibration.
\newblock {\em SIAM J. Imaging Sci.}, 3(2):199--231.

\bibitem[Cohen et~al., 2003]{cohen2003harmonic}
Cohen, A., Dahmen, W., Daubechies, I., and DeVore, R. (2003).
\newblock Harmonic analysis of the space \textup{BV}.
\newblock {\em Rev. Mat. Iberoam.}, 19(1):235--263.

\bibitem[Daubechies, 1992]{daubechies1992ten}
Daubechies, I. (1992).
\newblock {\em Ten lectures on wavelets}, volume~61.
\newblock Society for Industrial and Applied Mathematics, Philadelphia.

\bibitem[del {\'A}lamo et~al., 2018]{mTV}
del {\'A}lamo, M., Li, H., and Munk, A. (2018).
\newblock Frame-constrained total variation regularization for white noise
  regression.
\newblock {\em arXiv preprint arXiv:1807.02038}.

\bibitem[Dong et~al., 2011]{dong2011automated}
Dong, Y., Hinterm{\"u}ller, M., and Rincon-Camacho, M.~M. (2011).
\newblock Automated regularization parameter selection in multi-scale total
  variation models for image restoration.
\newblock {\em J. Math. Imaging Vision}, 40(1):82--104.

\bibitem[Donoho, 1995]{donoho1995nonlinear}
Donoho, D.~L. (1995).
\newblock Nonlinear solution of linear inverse problems by wavelet--vaguelette
  decomposition.
\newblock {\em Appl. Comput. Harmon. Anal.}, 2(2):101--126.

\bibitem[Donoho and Johnstone, 1998]{donoho1998minimax}
Donoho, D.~L. and Johnstone, I.~M. (1998).
\newblock Minimax estimation via wavelet shrinkage.
\newblock {\em Ann. Statist.}, 26(3):879--921.

\bibitem[Evans and Gariepy, 2015]{Evans}
Evans, L.~C. and Gariepy, R.~F. (2015).
\newblock {\em Measure theory and fine properties of functions}.
\newblock CRC press.

\bibitem[Frick et~al., 2012]{frick2012}
Frick, K., Marnitz, P., and Munk, A. (2012).
\newblock Statistical multiresolution \textup{Dantzig estimation in imaging:
  Fundamental concepts and }algorithmic framework.
\newblock {\em Electron. J. Stat.}, 6:231--268.

\bibitem[Frick et~al., 2013]{frick2013statistical}
Frick, K., Marnitz, P., and Munk, A. (2013).
\newblock Statistical multiresolution estimation for variational imaging: With
  an application in Poisson-biophotonics.
\newblock {\em J. Math. Imaging Vision}, 46(3):370--387.

\bibitem[Gin{\'e} and Nickl, 2015]{gine2015mathematical}
Gin{\'e}, E. and Nickl, R. (2015).
\newblock {\em Mathematical foundations of infinite-dimensional statistical
  models}, volume~40.
\newblock Cambridge University Press.

\bibitem[Goldenshluger and Lepskii, 2014]{goldenshluger}
Goldenshluger, A. and Lepskii, O. (2014).
\newblock On adaptive minimax density estimation on $\mathbb{R}^d$.
\newblock {\em Probab. Theory Related Fields}, 159(3-4):479--543.

\bibitem[Grama and Nussbaum, 1998]{grama1998asymptotic}
Grama, I. and Nussbaum, M. (1998).
\newblock Asymptotic equivalence for nonparametric generalized linear models.
\newblock {\em Probab. Theory Related Fields}, 111(2):167--214.

\bibitem[Grasmair et~al., 2018]{grasmair2015}
Grasmair, M., Li, H., and Munk, A. (2018).
\newblock Variational multiscale nonparametric regression: smooth functions.
\newblock {\em Ann. Inst. Henri Poincar\`e Probab. Stat.}, 54(2):1058--1097.

\bibitem[Haltmeier, 2013]{haltmeier2013inversion}
Haltmeier, M. (2013).
\newblock Inversion of circular means and the wave equation on convex planar
  domains.
\newblock {\em Comput. Math. Appl.}, 65(7):1025--1036.

\bibitem[H{\"a}rdle et~al., 2012]{hardle2012wavelets}
H{\"a}rdle, W., Kerkyacharian, G., Picard, D., and Tsybakov, A. (2012).
\newblock {\em Wavelets, approximation, and statistical applications}, volume
  129.
\newblock Springer Science \& Business Media.

\bibitem[Lepskii, 1991]{lepskii1991problem}
Lepskii, O. (1991).
\newblock On a problem of adaptive estimation in \textup{Gaussian} white noise.
\newblock {\em Theory Probab. Appl.}, 35(3):454--466.

\bibitem[Lepskii, 2015]{lepskii2015}
Lepskii, O. (2015).
\newblock Adaptive estimation over anisotropic functional classes via oracle
  approach.
\newblock {\em Ann. Statist.}, 43(3):1178--1242.

\bibitem[Malitsky and Pock, 2018]{malitsky2018first}
Malitsky, Y. and Pock, T. (2018).
\newblock A first-order primal-dual algorithm with linesearch.
\newblock {\em SIAM J. Optim.}, 28(1):411--432.

\bibitem[Mammen and van~de Geer, 1997]{mammen1997}
Mammen, E. and van~de Geer, S. (1997).
\newblock Locally adaptive regression splines.
\newblock {\em Ann. Statist.}, 25(1):387--413.

\bibitem[Math{\'e} and Pereverzev, 2003]{mathe2003geometry}
Math{\'e}, P. and Pereverzev, S.~V. (2003).
\newblock Geometry of linear ill-posed problems in variable \textup{Hilbert}
  scales.
\newblock {\em Inverse Problems}, 19(3):789.

\bibitem[Meister, 2011]{meister2011asymptotic}
Meister, A. (2011).
\newblock Asymptotic equivalence of functional linear regression and a white
  noise inverse problem.
\newblock {\em Ann. Statist.}, 39(3):1471--1495.

\bibitem[Meyer, 2001]{Meyer}
Meyer, Y. (2001).
\newblock {\em Oscillating patterns in image processing and nonlinear evolution
  equations: the fifteenth Dean Jacqueline B. Lewis memorial lectures},
  volume~22.
\newblock American Mathematical Society.

\bibitem[Munk et~al., 2005]{munk2005difference}
Munk, A., Bissantz, N., Wagner, T., and Freitag, G. (2005).
\newblock On difference-based variance estimation in nonparametric regression
  when the covariate is high dimensional.
\newblock {\em J. R. Stat. Soc. Ser. B. Stat. Methodol.}, 67(1):19--41.

\bibitem[Natterer, 1986]{natterer1986mathematics}
Natterer, F. (1986).
\newblock {\em The mathematics of computerized tomography}, volume~32.
\newblock Siam.

\bibitem[Nemirovski, 1985]{nemirovskii1985nonparametric}
Nemirovski, A. (1985).
\newblock Nonparametric estimation of smooth regression functions.
\newblock {\em Izv. Akad. Nauk. SSR Teckhn. Kibernet.}, 3:50--60.

\bibitem[Nesterov and Nemirovsky, 1994]{nesterov1994interior}
Nesterov, Y. and Nemirovsky, A. (1994).
\newblock Interior-point polynomial methods in convex programming.
\newblock {\em Stud. Appl. Math.}, 13.

\bibitem[Nirenberg, 1959]{nirenberg1959}
Nirenberg, L. (1959).
\newblock On elliptic partial differential equations.
\newblock In {\em Il principio di minimo e sue applicazioni alle equazioni
  funzionali}, pages 1--48. Springer.

\bibitem[Proksch et~al., 2018]{proksch2018multiscale}
Proksch, K., Werner, F., and Munk, A. (2018).
\newblock Multiscale scanning in inverse problems.
\newblock {\em Ann. Statist.}, 46(6B):3569--3602.

\bibitem[Reiss, 2008]{reiss2008}
Reiss, M. (2008).
\newblock Asymptotic equivalence for nonparametric regression with multivariate
  and random design.
\newblock {\em Ann. Statist.}, 36(4):1957--1982.

\bibitem[Rudin et~al., 1992]{ROF}
Rudin, L.~I., Osher, S., and Fatemi, E. (1992).
\newblock Nonlinear total variation based noise removal algorithms.
\newblock {\em Phys. D}, 60(1-4):259--268.

\bibitem[Sadhanala et~al., 2016]{sadhanala2016total}
Sadhanala, V., Wang, Y.-X., and Tibshirani, R.~J. (2016).
\newblock Total variation classes beyond 1d: Minimax rates, and the limitations
  of linear smoothers.
\newblock In {\em Advances in Neural Information Processing Systems}, pages
  3513--3521.

\bibitem[Scherzer et~al., 2009]{scherzer2009variational}
Scherzer, O., Grasmair, M., Grossauer, H., Haltmeier, M., and Lenzen, F.
  (2009).
\newblock {\em Variational methods in imaging}.
\newblock Springer.

\bibitem[Schmidt-Hieber et~al., 2013]{schmidt2013multiscale}
Schmidt-Hieber, J., Munk, A., and D{\"u}mbgen, L. (2013).
\newblock Multiscale methods for shape constraints in deconvolution: Confidence
  statements for qualitative features.
\newblock {\em Ann. Statist.}, 41(3):1299--1328.

\bibitem[Spokoiny, 2002]{spokoiny2002variance}
Spokoiny, V. (2002).
\newblock Variance estimation for high-dimensional regression models.
\newblock {\em J. Multivariate Anal.}, 82(1):111--133.

\bibitem[Stein and Weiss, 1971]{stein2016introduction}
Stein, E.~M. and Weiss, G. (1971).
\newblock {\em Introduction to Fourier analysis on Euclidean spaces},
  volume~32.
\newblock Princeton University Press.

\bibitem[Triebel, 1983]{TriebelI}
Triebel, H. (1983).
\newblock {\em Theory of Function Spaces}, volume~78.
\newblock Monographs in Mathematics, Birkh\"auser-Verlag, Basel.

\bibitem[Tsybakov, 2008]{Tsybakov}
Tsybakov, A.~B. (2008).
\newblock {\em Introduction to Nonparametric Estimation}.
\newblock Springer Publishing Company.

\bibitem[Wahba, 1977]{wahba1977practical}
Wahba, G. (1977).
\newblock Practical approximate solutions to linear operator equations when the
  data are noisy.
\newblock {\em SIAM J. Numer. Anal.}, 14(4):651--667.

\end{thebibliography}
%


\end{document}